\documentclass[11pt]{article}

\usepackage{amsmath,amsthm,amssymb,exscale,latexsym,bbm, a4wide}
\usepackage{mathtools, graphicx, color}
\usepackage{natbib,verbatim}
\usepackage[normalem]{ulem} 

\usepackage[latin1]{inputenc}
\usepackage[british, english]{babel}

\renewcommand{\Pr}{\mathbb{P}}
\newcommand{\N}{\mathbb{N}}

\newcommand{\R}{\mathbb{R}}
\newcommand{\E}{\mathbb{E}}

\newcommand{\FF}{\mathcal{F}}

\DeclareMathOperator{\Var}{Var}

\newcommand\unnumberedfootnote[1]{ %
        \let\temp=\thefootnote %
        \renewcommand{\thefootnote}{}%
        \footnote{#1}%
        \let\thefootnote=\temp%
        \addtocounter{footnote}{-1}}

\newcommand{\abs}[1]{\lvert#1\rvert} 

\newcommand{\ind}[1]{\mathbbm{1}_{\{#1\}}} 

\newtheorem{theorem}{Theorem}
\newtheorem{proposition}{Proposition}[section]

\newtheorem{lemma}[proposition]{Lemma}
\theoremstyle{definition}

\newtheorem{remark}[proposition]{Remark}
\newtheorem{definition}[proposition]{Definition}
\numberwithin{equation}{section}

\begin{document}

\title{\LARGE Asymptotics of a Brownian ratchet for Protein
  Translocation}

\author{\sc by Andrej Depperschmidt\thanks{supported by the BMBF,
    Germany, through FRISYS (Freiburg Initiative for Systems biology),
    Kennzeichen 0313921.} and Peter
  Pfaffelhuber$\mbox{}^{\ast,}$\thanks{\emph{Corresponding author,}
    University of Freiburg, Abteilung Mathematische Stochastik,
    Eckerstr. 1, D -- 79104 Freiburg, Germany, email: {\tt
      p.p@stochastik.uni-freiburg.de}, Tel:
    +49--761 203 5667, Fax: +49--761 203 5661} \\[2ex]
  \emph{Albert-Ludwigs University Freiburg} \vspace*{-5ex}} \date{}


\maketitle
\unnumberedfootnote{\emph{AMS 2000 subject classification.} 92C37
  (Primary) 60J65, 60G55, 60K05 (Secondary).}

\unnumberedfootnote{\emph{Keywords and phrases.} Brownian motion,
  reflected Brownian motion, cumulative process, protein
  translocation, ratchet mechanism}

\maketitle

\begin{abstract} 
  \noindent Protein translocation in cells has been modelled by
  \emph{Brownian ratchets}.  In such models, the protein diffuses
  through a nanopore. On one side of the pore, ratcheting molecules
  bind to the protein and hinder it to diffuse out of the pore. We
  study a Brownian ratchet by means of a reflected Brownian motion
  $(X_t)_{t\geq 0}$ with a changing reflection point $(R_t)_{t\geq
    0}$. The rate of change of $R_t$ is $\gamma(X_t-R_t)$ and the new
  reflection boundary is distributed uniformly between
  $R_{t-}$ and $X_t$.
  The asymptotic speed of the ratchet scales with $\gamma^{1/3}$ and
  the asymptotic variance is independent of $\gamma$.
\end{abstract}
  
\bibliographystyle{chicago}

\section{Introduction}
\label{sec:introduction}
Brownian motion is a model of thermal fluctuations of small
particles. If a particle moves according to such fluctuations, it is
known to undergo undirected movement, i.e.\ Brownian motion is a
martingale. However, the idea to use random fluctuations in order to
force a particle in one direction is tempting and led to the paradox
of the Brownian ratchet as introduced by \citet[Chapter
46]{FeynmanLeightonSWands1963}.

Quantitative models of Brownian ratchets for biological mechanisms
were introduced in \citet{SimonPeskinOster1992} and
\citet{PeskinOdellOster:1993}. While the \emph{polymerisation ratchet}
serves as a model of growth of polymers against a barrier, we focus on
the \emph{translocation ratchet}, a model for protein transport.
Here, a polymer moves \emph{in} and \emph{out} through a nanopore by
thermal fluctuations. On the \emph{in}-side of the nanopore, molecules
may bind to the polymer and bound sites along the polymer are
forbidden to move through the nanopore; see Figure \ref{fig:illu} for
an illustration.

A prominent example of a translocation ratchet was suggested on the
basis of empirical data for Prepro-$\alpha$ Factor as protein and BiP
as the ratcheting molecule by \citet{Matlack1999}: after translation,
proteins have to be transported into the lumen of the endoplasmatic
reticulum (ER) where they are e.g.\ cut and folded in order to
function properly.  In general, such ratcheting molecules can be
chaperones which bind to the translocated protein \emph{in}-side the
ER lumen. The resulting translocation ratchet has been studied and
compared to data from \citet{Matlack1999} by
\citet{LiebermeisterRapoportHeinrich:2001} and \cite{Elston2002}. In
these models, a discrete set of sites \emph{in}-side the ER lumen is
either bound or unbound with BiP. The protein diffuses to either side
of the nanopore with equal chances. Each unbound site of the protein
may become bound by BiP and each bound BiP molecule may dissociate
from the prepro-$\alpha$ Factor. Finally, when the protein is
completely \emph{in}-side the ER lumen, it is released.

In mathematical models of \emph{translocation ratchets}, several
parameters have to be specified: the distance of sites along the
protein which can be either bound or unbound with ratcheting
molecules; the rate and location of association and rate
of dissociation of ratcheting molecules. The case of a large
association rate and a fixed distance of possible ratcheting sites was
studied in \cite{BudhirajaFricks2006}. (This ratchet is similar to the
polymerisation ratchet of \citealp{PeskinOdellOster:1993}.)
\citeauthor{BudhirajaFricks2006} obtain a Law of Large Numbers and a
Central Limit Theorem for the speed of protein translocation if the
protein diffuses through a Brownian motion with drift.  In the present
paper, we study a translocation ratchet for a continuum of possible
ratcheting sites, in the limit of small ratcheting molecules, small
dissociation rates and long proteins. In this model, we can compute a
Law of Large Numbers (Theorem \ref{T:LLN}) and a Central Limit Theorem
(Theorem \ref{T:CLT}) for the speed of protein translocation.

\begin{figure}
  \hspace{5cm} (A) \hspace{7cm} (B)
  \begin{center}
    \parbox{6cm}{\includegraphics[width=8cm]{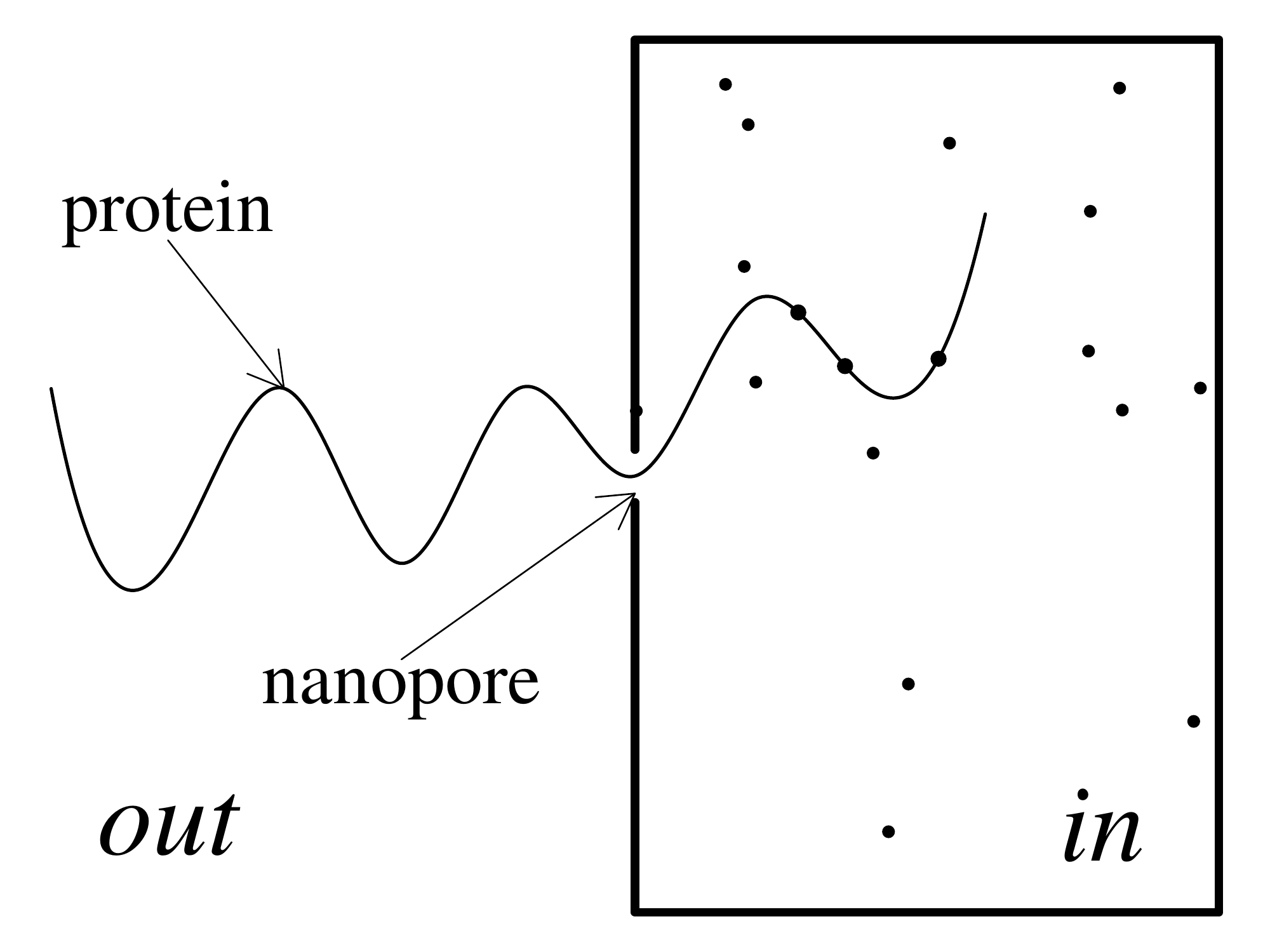}}
    \hspace{3cm}
    \parbox{4cm}{
      \includegraphics[width=5cm]{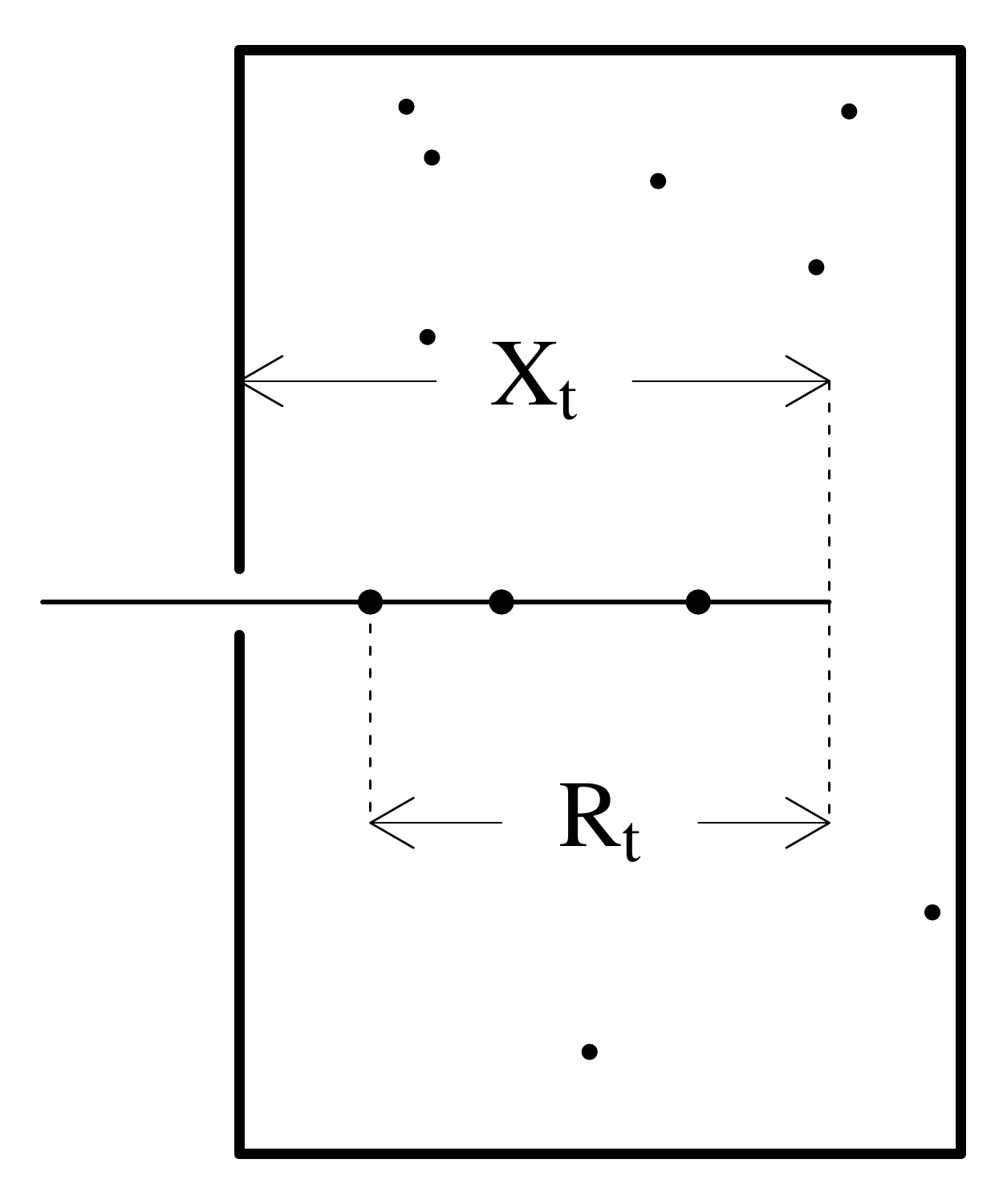}}
  \end{center}
  \caption{\label{fig:illu} (A) Illustration of the translocation
  ratchet. A protein moves through a nanopore by thermal
  fluctuations. On the \emph{in}-side of the pore, molecules may bind
  to the protein which prevent the protein from moving \emph{out}
  again. (B) The state of the protein at time $t$. The total length of
  the protein on the \emph{in}-side is $X_t$ whereas the last bound
  molecule is located at $R_t$.}
\end{figure}

~

The paper is organised as follows: in Section \ref{sec:model-results}
we introduce our model and give the main results, the Law of Large
Numbers (Theorem \ref{T:LLN}) and the Central Limit Theorem (Theorem
\ref{T:CLT}). In Section \ref{sec:appl} we interpret our results with
respect to existing literature on ratchet models. In Sections
\ref{sec:graph-repr}, \ref{sec:invar-distr} and \ref{sec:regen-struct}
we provide the three main techniques used for the proof of Theorems
\ref{T:LLN} and \ref{T:CLT}; see also Remark \ref{rem:thm}. We
conclude with a formal proof of both theorems in Section
\ref{sec:proof}.

\section{Model and Results}
\label{sec:model-results}
We study a \emph{translocation ratchet} similar to the model by
\citet{LiebermeisterRapoportHeinrich:2001} and \cite{Elston2002} using
the following assumptions: (i) the protein moves \emph{in} and
\emph{out} with equal probabilities; (ii) the protein movement is
reflected at bound ratcheting molecules; (iii) the protein has a
continuum of sites to which ratcheting molecules can bind; (iv) the
dissociation rate of ratcheting molecules from the protein is much
smaller than their binding rate to the protein; (v) the ratcheting
molecules are infinitely small; (vi) the protein is infinitely long.

In mathematical terms, we consider a time-homogeneous Markov dynamics
$(\mathcal X, \mathcal R) = (X_t, R_t)_{t\geq 0}$, starting in $(X_0,
R_0) = (x,0), x\geq 0$. Here, at time $t$, $X_t$ is the length of the
protein on the \emph{in}-side and $R_t$ is the distance of the
molecule closest to the boundary as measured from the end point of the
protein that is inside the cell. The process $\mathcal X =
(X_t)_{t\geq 0}$ is reflected Brownian motion, where the reflection
point at time $t$ is $R_t$. This reflection point $R_t$ increases with
$t$: given that the value of $(\mathcal X, \mathcal R)$ at time $t$ is
$(X_t , R_t)$, $R_t$ jumps to a uniformly chosen value $r \in [R_t ,
X_t ]$ at rate $\gamma(X_t-R_t) dr$ for some $\gamma>0$.
In other words, at rate $\gamma(X_t-R_t)$ a new ratcheting molecule
binds uniformly on $[R_t;X_t]$ which provides a new reflection point
for $\mathcal X$. By this dynamics, $R_t\leq X_t$ for all $t\geq 0$,
almost surely. We refer to $(\mathcal X, \mathcal R)$ as \emph{the
  $\gamma$-Brownian ratchet started in $x$}. If we want to stress the
dependence on $\gamma$, we write $(\mathcal X^\gamma, \mathcal
R^\gamma) = (X_t^\gamma, R_t^\gamma)_{t\geq 0}$ for the
$\gamma$-Brownian ratchet.  For a realization of the processes
$(\mathcal X, \mathcal R)$ see Figure \ref{fig:br}.

\begin{figure}
\begin{center}
  \includegraphics[width=12cm]{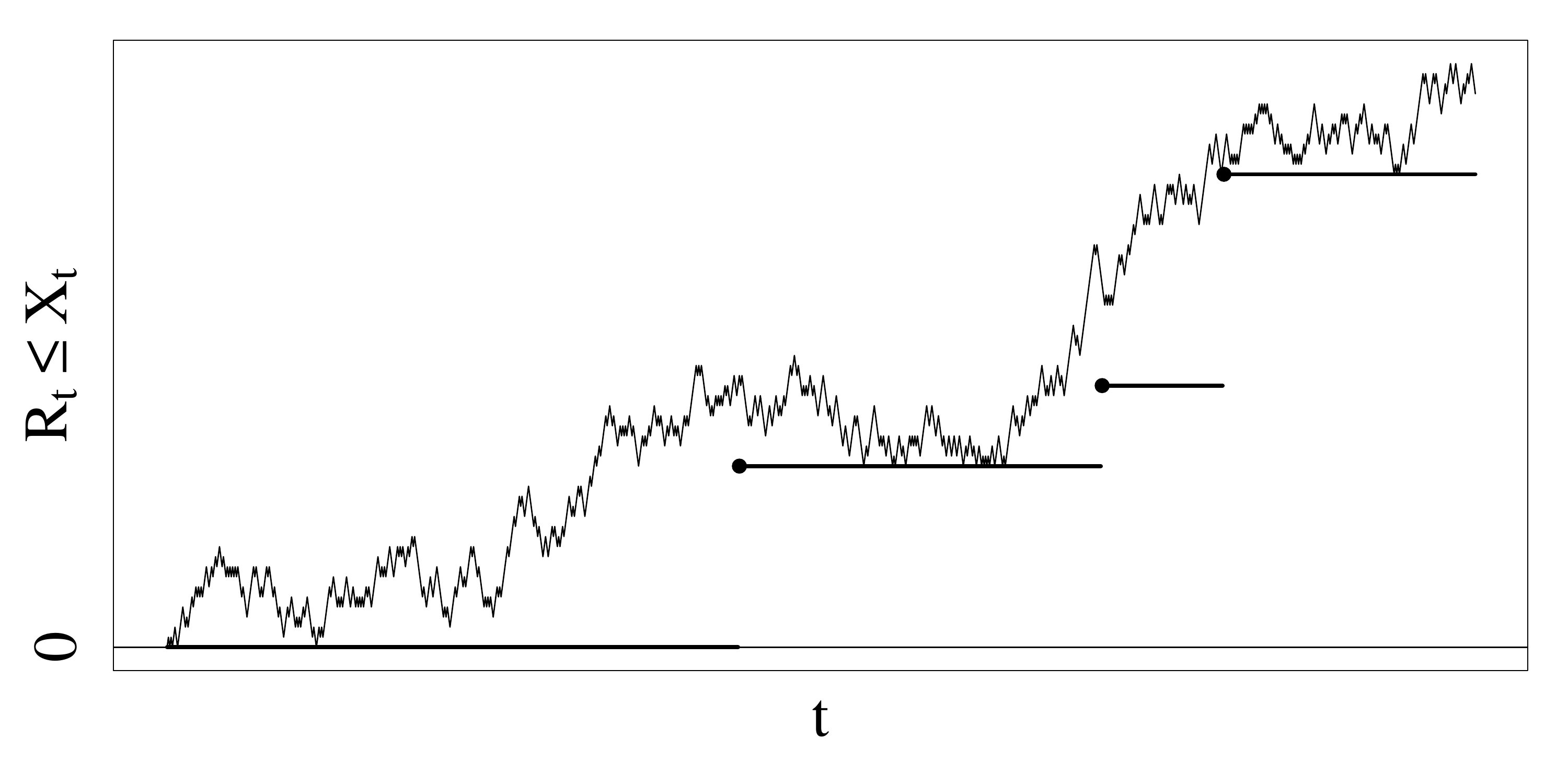}
\end{center}
\caption{\label{fig:br} One realization of the processes $(\mathcal X,
  \mathcal R)$, $\mathcal X = (X_t)_{t\geq 0}, \mathcal R=(R_t)_{t\geq
    0}$. The Brownian motion $\mathcal X$ is reflected at the boundary
  $\mathcal R$ which in turn jumps at time $t$ at rate proportional to
  $X_t - R_t$. }
\end{figure}

The rate of jumps of $\mathcal R$ interacts closely with the distance
of $\mathcal X$ and $\mathcal R$. Since $\mathcal R$ is non-decreasing
and $\mathcal R\leq \mathcal X$, the process $\mathcal X$ also tends
to grow. We immediately formulate our main results concerning the Law
of Large Numbers and the Central Limit Theorem for the speed of
$\mathcal X$:

\begin{theorem}[Law of Large Numbers]\label{T:LLN}
  Let $\mathcal X = (X_t)_{t\geq 0}$ be the $\gamma$-Brownian ratchet
  started in $x\geq 0$. Then,
  $$ \frac{X_t}{t} \xrightarrow{t\to\infty} C_\gamma $$ almost surely, where 
  \begin{align}\label{eq:cgamma}
    C_\gamma = \frac{\Gamma(2/3)}{\Gamma(1/3)}
    \Big(\frac{3\gamma}{4}\Big)^{1/3}
  \end{align}
  and $\Gamma(.)$ is the gamma function.
\end{theorem}

\begin{theorem}[Central Limit Theorem]\label{T:CLT}
  Let $\gamma>0$ and $\mathcal X = (X_t)_{t\geq 0}$ be the
  $\gamma$-Brownian ratchet started in $x\geq 0$, $C_\gamma$ as in
  \eqref{eq:cgamma} and $X$ some $N(0,1)$-distributed random
  variable. Then, there is $\sigma>0$ which is independent of $\gamma$
  such that
  $$ \frac{X_t - C_\gamma t}{\sigma \sqrt{t}} \xRightarrow{t\to\infty} X,$$
    where '$\xRightarrow{}$' denotes convergence in distribution.
\end{theorem}

\begin{remark}[Numerical values for $C_\gamma$ and $\sigma$]
  The numerical value for the speed of the $\gamma$-Brownian ratchet
  is $C_\gamma = \frac{\Gamma(2/3)}{\Gamma(1/3)}
  \Big(\frac{3\gamma}{4}\Big)^{1/3} \approx 0.459248 \gamma^{1/3}$.

  For the numerical value of $\sigma$, note that $X_t$ is a reflected
  Brownian motion for $\gamma=0$. Since $\sigma$ does not depend on
  $\gamma$ (as long as $\gamma>0$), it is tempting to conjecture that
  $\sigma=\sqrt{1-2/\pi}\approx 0.60281$, the asymptotic standard
  deviation for reflected Brownian motion. Clearly, this heuristics
  would require an interchange of limits, $t\to\infty$ and $\gamma\to
  0$. While we could not make this interchange rigorous, simulations
  support this conjecture for the value of $\sigma$.
\end{remark}

\begin{remark}[Main steps in the proofs]\label{rem:thm}
  The proof of Theorems \ref{T:LLN} and \ref{T:CLT} relies on several
  ingredients which we will provide in Sections \ref{sec:graph-repr},
  \ref{sec:invar-distr} and \ref{sec:regen-struct}, respectively.

  \sloppy First (see Section \ref{sec:graph-repr}), the process
  $(\mathcal X, \mathcal R)$ can be constructed graphically by a
  one-dimensional (non-reflecting) Brownian motion and an independent
  Poisson process on $[0;\infty)\times \mathbb R$ with intensity
  $\gamma$. From this graphical construction, it becomes clear that
  the speed of $X_t$, i.e., the limit of $\frac{X_t}{t}$, must be
  proportional to $\gamma^{1/3}$, if it exists. In addition, the
  graphical construction shows that $\Var[X_t]$ does not depend on
  $\gamma$; see Remark \ref{rem:graph1}.

  Second, to compute a candidate for the asymptotic speed $X_t/t$, we
  study the Markov chain $(X_{\tau_{n}} - R_{\tau_n},
  R_{\tau_n}-R_{\tau_{n-1}}, \tau_n-\tau_{n-1})_{n=1,2,...}$, where
  $\tau_0=0$ and $\tau_1, \tau_2,...$ are the jump times of $\mathcal
  R$ (see Section~\ref{sec:invar-distr}). In particular, we show that
  this Markov chain has a unique invariant distribution and (as shown
  in Section \ref{sec:proof}) the asymptotic speed is the ratio of
  expectations of $X_{\tau_{n}} - R_{\tau_n}$ and $\tau_n-\tau_{n-1}$
  under the invariant distribution, which are computed in
  Proposition~\ref{P:incrInv}.

  Third, we find a renewal structure for the process $(\mathcal X,
  \mathcal R)$ where renewal points $\rho_1, \rho_2,...$ are given by
  times $t$ where $X_t=R_t$ and between $\rho_n$ and $\rho_{n+1}$ a
  jump of $\mathcal R$ has occurred (see
  Section~\ref{sec:regen-struct}). Using this renewal structure, we can show 
  existence for the speed of $X_t$ and use a Central Limit Theorem for
  cumulative processes in order to prove Theorem~\ref{T:CLT}.
\end{remark}

\section{Applications in biology and extensions}
\label{sec:appl}
We describe possible extensions of the $\gamma$-Brownian ratchet some
of which already appeared in the literature. Moreover, we give
biological interpretations of our findings.

\begin{remark}[Review of published ratchet models and
  extensions]\label{rem:31}
  Modelling protein translocation by ratcheting mechanisms has started
  with \cite{SimonPeskinOster1992} and \cite{PeskinOdellOster:1993}.
  We review several features of published models for protein
  translocation 
  and hint to extensions of our mathematical model. Throughout, we
  assume that $X_t$ is the length of the protein \emph{in}-side and
  $R_t$ is the position of the reflection point at time $t$.

  ~

  \noindent
  \emph{Dissociation of ratcheting molecules}: In the original
  approach of \cite{SimonPeskinOster1992}, there is a finite set of
  equally spaced ratcheting sites along the protein. There are two
  rates which describe binding and dissociation of ratcheting
  molecules from the protein. In particular, the case of infinite
  rates with a constant ratio can be studied, which leads to an
  effective speed by assuming that each ratcheting site has a certain
  probability of being bound independent of all others. This is also
  the limiting case of fast binding and unbinding as studied in
  \citet{Amb2005}. These approaches lead to the following mathematical
  description:
  \begin{quote}
    (i) If $\mathcal R_t$ is the set of positions of bound ratcheting
    molecules at time $t$, let the reflection point of $X_t$ be $R_t
    := \max \mathcal R_t$. In addition, a point $r$ with $0\leq r\leq
    X_t$ is added at rate $\rho dr$ and an existing point
    $r\in\mathcal R_t$ is taken from $\mathcal R_t$ at a rate
    $\sigma$. Here, $\sigma$ describes the rate of dissociation of
    ratcheting molecules and the $\gamma$-Brownian ratchet is the
    special case $\sigma=0$.
  \end{quote}

  \noindent
  \emph{Protein movement against a force}: Usually, proteins are
  present in folded states inside a cell. Travelling through a
  nanopore requires unfolding of these proteins. In particular, the
  part of the protein on the \emph{out}-side might still be in a
  folded state while the part on the \emph{in}-side is
  unfolded. Hence, moving \emph{in} takes place against a force, as
  modelled by \citet{LiebermeisterRapoportHeinrich:2001}; see also
  \cite{BudhirajaFricks2006}. The mathematical description of such
  ideas extends the $\gamma$-Brownian ratchet as follows:
  \begin{quote}
    (ii) Between reflection events, $(X_t)_{t\geq 0}$ is a stochastic
    process with continuous paths, not necessarily a Brownian
    motion. In particular, if proteins have to unfold \emph{out}-side
    the nanopore, $(X_t)_{t\geq 0}$ is best chosen as a reflected
    Brownian motion with negative drift.
  \end{quote}

  \noindent
  \emph{State-dependent binding rate}: The binding of ratcheting
  molecules to the protein might depend on various parameters. In the
  model of \cite{LiebermeisterRapoportHeinrich:2001} ratcheting sites
  along the protein are occupied only if they are close to the
  nanopore. The physical reason for this preferred binding close to
  the nanopore is an interaction with the protein forming the
  nanopore. Such a phenomenon has been described for the translocation
  of Prepro-$\alpha$ Factor (the protein) with BiP (the ratcheting
  molecule) and Sec61p as the protein forming the nanopore
  \citep{Matlack1999}. A compromise between ratcheting molecules which
  bind with uniform rates along the protein and only in proximity to
  the nanopore was considered by \citet{Elston2002}.  Finally, binding
  of ratcheting molecules might depend on the amino acid sequence of
  the protein as described in \citet{pmid19102558}. We suggest the
  following extension of the $\gamma$-Brownian ratchet:
  \begin{quote}
    (iii) Given $X_t$ and $R_t$, let the reflection point change to
    $r$, $R_t\leq r \leq X_t$ at rate $\rho(d(X_t-r))$ for some
    $\sigma$-finite measure $\rho$ on $\mathbb R_+$. The
    $\gamma$-Brownian ratchet is the special case that $\rho$ is
    Lebesgue measure on $\mathbb R_+$.
  \end{quote}
\end{remark}

Some special cases of the extensions (i), (ii), (iii) are straight
forward to analyse. Consider the extension (i) as an example. As will
become clear in Section \ref{sec:graph-repr} the average time between
jumps of $\mathcal R$ in the $\gamma$-Brownian ratchet scales with
$\gamma^{-2/3}$. This implies that if $\sigma\ll\gamma^{2/3}$ the
extension (i) is approximately equal to the $\gamma$-Brownian ratchet
but in other cases, the speed of the ratchet in (i) remains to be
solved. Another example is the extension (iii) in the case $\rho =
\gamma\cdot \delta_0$ for some $\gamma>0$, i.e., binding of ratcheting
molecules to the protein occur at constant rate directly at the
nanopore. For the case of BiP as ratcheting molecule and
Prepro-$\alpha$ as protein, parameters estimated in \citet{Elston2002}
indicate that this is a realistic scenario. In this case, jump times
of $\mathcal R = (R_t)_{t\geq 0}$ are a rate-$\gamma$ Poisson
process. At each jump time, we set $R_t = X_t$ which is also the new
reflection point of the Brownian motion. Therefore, jump times are
renewal points of such a Brownian ratchet. Moreover, starting a
reflected Brownian motion $(B_t)_{t\geq 0}$ at 0 and if $T$ is an
exponential time with rate $\gamma$, it can be computed
\citep[p.~333]{BorodinSalminen:2002} that $B_T$ is exponentially
distributed with rate $(2\gamma)^{1/2}$. This leads to an asymptotic
velocity of $(2\gamma)^{-1/2}\cdot\gamma = (\gamma/2)^{1/2}$. A
Central Limit Theorem can be proved as well.

\begin{remark}[Biological interpretation of our results]
  In our model we use the asymptotics of an infinitely long protein
  and study properties of the velocity of the Brownian ratchet. In
  practice more relevant are of course a finitely long protein and the
  time the protein is completely on the \emph{in}-side. For long
  proteins both approaches are similar: if $T_x:=\inf\{t: X_t\geq x\}$
  in our model, $T_x/x \approx 1/C_\gamma$ with $C_\gamma$ from
  \eqref{eq:cgamma} for large $x$.

  Consider the case that the rate $\gamma$ for binding of ratcheting
  molecules to the protein, is proportional to the concentration $a$
  of ratcheting molecules on the \emph{in}-side of the nanopore. As
  Theorem \ref{T:LLN} shows, the speed of translocation, $X_t/t$, is
  mainly determined by $\gamma^{1/3}$ which is proportional to
  $a^{1/3}$. That is, to double the speed of translocation requires an
  eight times higher concentration of ratcheting molecules. In
  contrast, consider a ratchet model described at the end of Remark
  \ref{rem:31} where ratcheting molecules bind to the protein
  preferably in close proximity to the nanopore. We observed that the
  speed of the ratchet is $(\gamma/2)^{1/2}$, so the speed of the
  ratchet is proportional to $a^{1/2}$. This means that the latter
  ratchet model uses the existing ratcheting molecules more efficient
  (i.e.\ protein translocation is faster) if the concentration $a$ is
  large but the $\gamma$-Brownian ratchet is more efficient when $a$
  is small. It will be interesting to see if real biological systems
  tend to behave like the more efficient model.
\end{remark}

\begin{remark}[Edwards model]
  One mathematical model similar in spirit to the Brownian ratchet is
  the Edwards model
  \citep[see][]{vdHofstad_dHollander_Koenig:1997}. Let $(B_t)_{t\ge
    0}$ be standard Brownian motion and let $P$ denote its
  distribution on path space. Edwards' model is a
  transformed measure on the path space defined for $\gamma>0$ by
  \begin{align*}
    \frac{dP_T^\gamma}{dP} = \frac1{Z_T^\gamma}\exp\Big\{-\gamma \int_\R
    L(T,x)^2\,dx, \Big\},\;  T  \ge 0, 
  \end{align*}
  where $L(T,x)$ is the local time in $x$ up to time $T$ and
  $Z_T^\gamma$ is the normalising constant. It can be seen that
  $\int_\R L(T,x)^2\,dx$ is the self intersection local time. Thus,
  the new measure discourages self intersections of paths. For the
  Brownian ratchet, such intersections are also discouraged by the
  jumping reflection boundary. 
  
  The weak Law of Large Numbers is proven in
  \citep{Westwater:1985}. In \cite{vdHofstad_dHollander_Koenig:1997}
  the Central Limit Theorem in the following form is proven: For every
  $\gamma \in (0,\infty)$
  \begin{align*}
    \lim_{T \to \infty} P_T^\gamma \biggl( \frac{\abs{B_T} -
      C_\gamma^* T}{\sigma^* \sqrt{T}} \le C\biggr) =
    \frac{1}{\sqrt{2\pi}} \int_{-\infty}^C e^{-x^2/2}, \quad \text{
      for all } \; C \in \R.
  \end{align*}
  As in the case of the Brownian ratchet, the asymptotic speed is of
  the form $C_\gamma^* = b^*\gamma^{1/3}$ for some absolute constant
  $b^*$. In addition, $\sigma^*$ does not depend on $\gamma$ as well.
  The constants $b^*$ and $\sigma^*$ can be expressed in terms of the
  largest eigenvalue of a certain Sturm-Liouville operator. In
  \cite{vanderHofstad:1998}, Theorem 3, it is shown that $b^* \in
  [1.104, 1.124]$ and $\sigma^* \in[0.60,0.66]$.  In the case of
  Brownian ratchet we have $C_\gamma=b\gamma^{1/3}$ where $b \approx
  0.459248$ and the conjectured value of $\sigma$ is $\sigma \approx
  0.60281$.  
\end{remark}

\section{Graphical construction and some applications}
\label{sec:graph-repr}
\begin{figure}
\begin{center}
  \includegraphics[width=12cm]{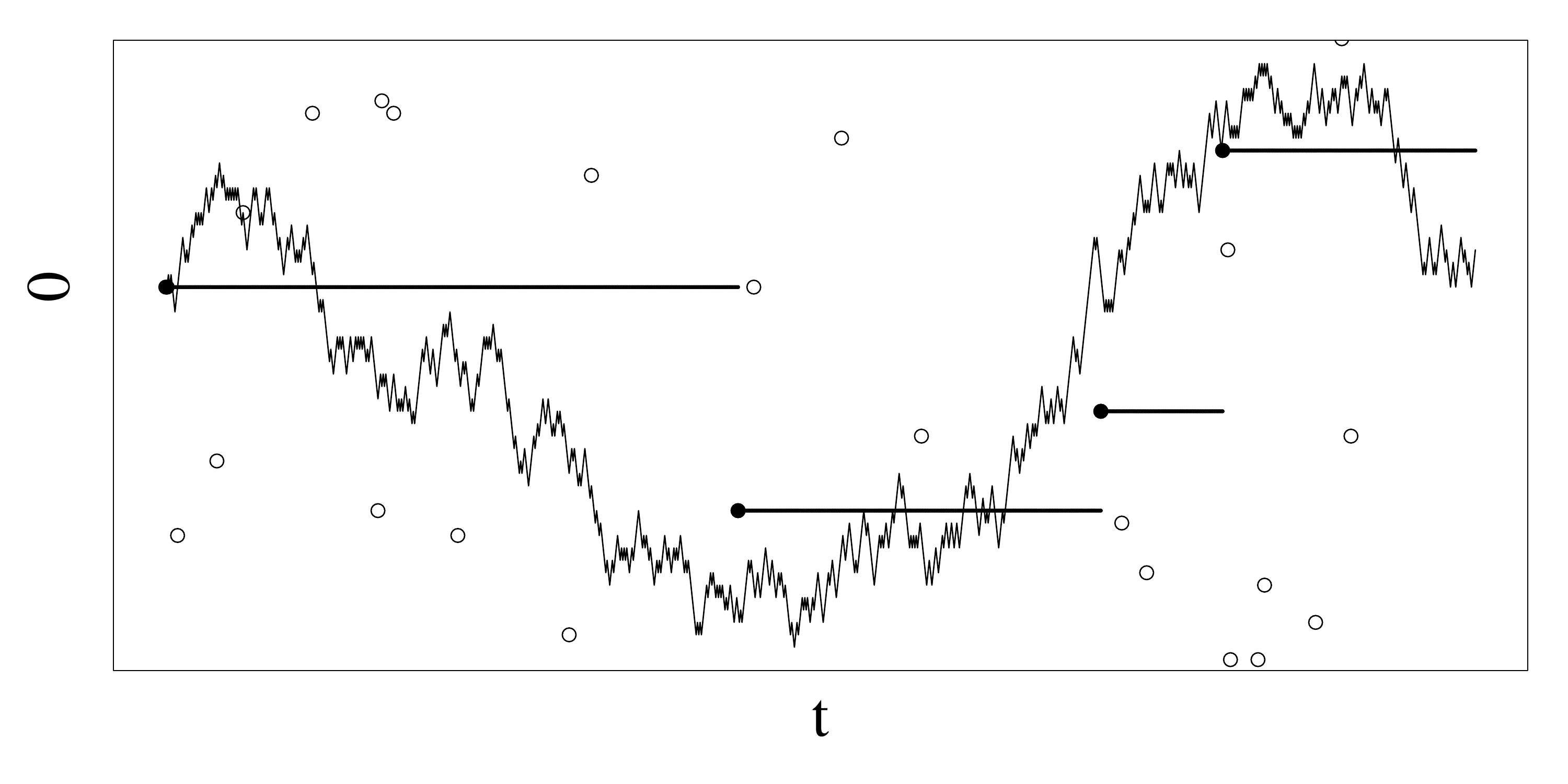}
\end{center}
\caption{\label{fig:graph} One realization of the graphical
  construction. The shown realization of the graphical construction
  leads to the same path of $(\mathcal X, \mathcal R)$ as shown in
  Figure~\ref{fig:br}.  }
\end{figure}
In this section we give a graphical construction of the Brownian
ratchet; see also Figure~\ref{fig:graph}. From that construction we
deduce a useful scaling (Proposition~\ref{P:scale}) that will be
needed in the sequel.

\subsection{The construction}
\begin{definition}\label{def:graph}
  For $\gamma>0$ let $N^\gamma$ be a Poisson point process on
  $[0;\infty)\times \R$ with intensity measure $\gamma
  \lambda^2(dt,dx)$, where $\lambda^2$ denotes the Lebesgue measure on
  $[0;\infty)\times \R$. Let $\mathcal B = (B_t)_{t\ge 0}$ be an independent 
  Brownian motion on $\R$ starting in $x\in\mathbb R$.  We define
  $\tilde \tau_0, S_0,\tilde \tau_1, S_1,\tilde \tau_2, S_2,\ldots$
  recursively by $\tilde\tau_0=0$, $S_0=0$,
  \begin{align}\label{eq:tildetau}
    \tilde\tau_{n+1} : = \inf
    \bigl\{t>\tilde\tau_{n}:\bigl(\{t\}\times [B_t \wedge
    S_{n},B_t\vee S_{n}] \bigl) \cap N^\gamma \ne \emptyset \bigr\}
  \end{align}
  and
  \begin{align} \label{eq:Sn}
    S_{n+1} & : =
    \begin{cases}
      \text{second coordinate of the unique element of } \\
      \big(\{\tilde\tau_{n+1}\} \times [B_{\tilde\tau_{n+1}} \wedge
      S_{n},B_{\tilde\tau_{n+1}}\vee S_{n}]\big) \cap N^\gamma.
    \end{cases}
  \end{align}
  That is, $\tau_{n+1}$ is the first time after $\tau_n$ when there is
  a point of the Poisson process between the graph of the Brownian
  motion and $S_n$, and $S_{n+1}$ is the second coordinate (the space
  coordinate) of that point.

  For $ t \in [\tilde\tau_n,\tilde\tau_{n+1})$ we set
  \begin{align} \label{eq:52} \widetilde R_t = \sum_{i=1}^n \abs{S_i-
      S_{i-1}}\quad \text{and} \quad \widetilde X_t = \widetilde R_t +
    \abs{B_t-S_n}
  \end{align}
  and $(\widetilde{\mathcal X}, \widetilde{\mathcal R}) := (\widetilde
  X_t, \widetilde R_t)_{t\geq 0}$.
\end{definition}

\begin{lemma}\label{l:key}
  The process $(\widetilde{\mathcal X}, \widetilde{\mathcal R})$ is a
  $\gamma$-Brownian ratchet started in $\abs x$.
\end{lemma}

\begin{proof}
  First note that $\widetilde{\mathcal R}$ is almost surely
  non-decreasing and $\widetilde X_t \geq \widetilde R_t$ for all
  $t\geq 0$ by construction. Consider any time
  $t\in[\tilde\tau_n;\tilde\tau_{n+1})$. The rate of occurrence of
  $\tilde\tau_{n+1}$ is $\gamma |B_t-S_n| = \gamma (\widetilde X_t -
  \widetilde R_t)$. Moreover, $\widetilde R_t$ jumps to $r$ uniformly
  chosen in $[\widetilde R_t; \widetilde X_t]$ if and only if
  $|S_{n+1}-S_n| = |r-\widetilde R_t|$. By homogeneity of the Poisson
  process $N^\gamma$, $r$ is uniform on $[\widetilde R_t, \widetilde
  X_t]$. In addition, $\widetilde X_t$ behaves like a Brownian motion,
  reflected at $\widetilde R_t$.
\end{proof}

\begin{definition}
  We denote that process $(\widetilde{\mathcal X}, \widetilde{\mathcal
    R})= (\widetilde X_t, \widetilde R_t)_{t\geq 0}$ constructed in
  Definition~\ref{def:graph} by the \emph{Brownian ratchet read off
    from $(\mathcal B, N^\gamma$)}. To stress the dependence on
  $\gamma$ we also write $(\widetilde{\mathcal X}^\gamma,
  \widetilde{\mathcal R}^\gamma) = (\widetilde X_t^\gamma, \widetilde
  R_t^\gamma)_{t\geq 0}$ for this process.
\end{definition}

\begin{remark}
  It is intuitively clear from the graphical construction that the
  long-time behaviour of $\mathcal X$, as given in
  Theorems~\ref{T:LLN} and \ref{T:CLT}, and $\mathcal R$, are
  identical. We will prove this fact in Proposition~\ref{prop:AsympA}.
\end{remark}

\subsection{Scaling property}
From the graphical construction we can deduce the following scaling
property.
\begin{proposition}\label{P:scale}
  Let $(\mathcal X^\gamma, \mathcal R^\gamma)=(X_t^\gamma,
  R_t^\gamma)_{t\geq 0}$ and $(\mathcal X^1, \mathcal R^1)=(X_t^1,
  R_t^1)_{t\geq 0}$ be Brownian ratchets with rates $\gamma>0$ and 1,
  respectively, both starting in $x=0$. Then,
  \begin{align} \label{eq:33}
    (X_t^\gamma,R_t^\gamma)_{t\ge 0} \stackrel{d}{=} \gamma^{-1/3} 
    \bigl(X_{\gamma^{2/3} t}^1,R_{\gamma^{2/3}t}^1\bigr)_{t\ge 0}.
  \end{align}
\end{proposition}
\begin{proof} 
  We use the same notation as in Definition~\ref{def:graph}. By
  Lemma~\ref{l:key} it suffices to show \eqref{eq:33} for the Brownian 
  ratchet $({\mathcal X}, {\mathcal R})$ read off from $\mathcal B$
  and $N^\gamma$. Setting
  \begin{align*}
    \begin{split}
      g:
      \begin{cases} \R^2 & \to \R^2 \\
        (t,x) & \mapsto (\gamma^{-2/3}t,\gamma^{-1/3}x).
      \end{cases}
    \end{split} 
  \end{align*}
  we find that $g(N^\gamma) \stackrel d = N^1$. Moreover, for the
  space-time version of the Brownian motion $\mathcal B$,
  i.e. $\widehat{\mathcal B} = (t,B_t)_{t\geq 0}$, started in $0$, we
  have $g(\widehat{\mathcal B}) := (g(t,B_t))_{t\geq 0} \stackrel d =
  \widehat{\mathcal B}$ by the Brownian rescaling. Applying $g$ to each
  space-time point we obtain a path of a Brownian ratchet based on
  $g(\widehat{\mathcal B}) \stackrel d = \widehat{\mathcal B}$ and
  $g(N^\gamma) \stackrel d = N^1$. In other words, \eqref{eq:33} holds
  for $(X_t^\gamma, R_t^\gamma)$ and we are done.
\end{proof}

\begin{remark}
  \label{rem:graph1}
  \begin{enumerate}
  \item From the scaling property of the Brownian ratchet,
    Proposition~\ref{P:scale}, we have 
    \begin{align*}
      \lim_{t \to\infty} \frac{\E[X_t^\gamma]}{t} & = \lim_{t \to
        \infty}\gamma^{-1/3} \frac{\E[X_{\gamma^{2/3}t}^1]}{t} =
      \lim_{t \to \infty}\gamma^{1/3}
      \frac{\E[X_{\gamma^{2/3}t}^1]}{\gamma^{2/3}t} =
      \gamma^{1/3}   \lim_{t \to\infty} \frac{\E[X_{t}^1]}{t} \\
      \intertext{and} \lim_{t \to\infty} \frac{\Var[X_t^\gamma]}{t} &
      = \lim_{t \to \infty}\gamma^{-2/3}
      \frac{\Var[X_{\gamma^{2/3}t}^1]}{t} = \lim_{t \to \infty}
      \frac{\Var[X_{\gamma^{2/3}t}^1]}{\gamma^{2/3}t} = \lim_{t
        \to\infty} \frac{\Var[X_{t}^1]}{t}.
    \end{align*}
    In particular, the second limit does not depend on $\gamma$. The
    existence of the limits will be proven in Section~\ref{sec:proof}
  \item Thanks to the scaling property of the Brownian ratchet, we may
    choose the most convenient value of $\gamma=\tfrac 12$ in most
    proofs below. Afterwards (i.e.\ in the proofs of Theorems~\ref{T:LLN} and
    \ref{T:CLT}) we use Proposition~\ref{P:scale} to  
    obtain results for general $\gamma$.
  \end{enumerate}
\end{remark}

\section{The Brownian ratchet at jump times}
\label{sec:invar-distr}
In this section we prove existence and uniqueness of the invariant
distribution of a Markov chain representing the Brownian ratchet at
jump times (see Definition~\ref{def:mc}). For this, we start in
Lemma~\ref{lem:Btaux} with studying the time to the next jump and the
increment of the Brownian ratchet during that time starting in $x$ in
expectation. We then derive properties of the tails of the waiting
time between jumps (Lemma~\ref{lem:taubound}). Afterwards we come to
existence (Proposition~\ref{P:exInv}) and uniqueness
(Proposition~\ref{P:coupl}) of an invariant distribution of the
Brownian ratchet at jump times. We conclude the section by computing
the waiting time to the next jump and increments for a Brownian
ratchet in equilibrium (Proposition~\ref{P:incrInv}). In the whole
section we restrict ourselves to the notationally convenient case
$\gamma=\tfrac 12$. Our results can easily be extended to general
$\gamma$ using the scaling property, Proposition~\ref{P:scale}.

\begin{definition}\label{def:mc}
  For the Brownian ratchet $(\mathcal X, \mathcal R)$, set $\tau_0=0$
  and denote the sequence of jump times of $\mathcal R$ by $\tau_1,
  \tau_2,\dots$. We set $(\mathcal Y, \mathcal W, \eta) := (Y_n, W_n,
  \eta_n)_{n=1,2,\dots}$ by
  \begin{align} \label{eq:53} 
    Y_n=X_{\tau_n}-R_{\tau_{n}}, \quad W_n = R_{\tau_n}-R_{\tau_{n-1}} \quad
    \text{and} \quad \eta_n=\tau_n-\tau_{n-1}. 
  \end{align}
  Observe that for any $k$, $(Y_n, W_n, \eta_n)_{n=k+1,k+2,\dots}$
  depends on $(Y_n, W_n, \eta_n)_{n=1,\dots,k}$ only through $Y_k$. In
  particular, $(\mathcal Y, \mathcal W, \eta)$ is a Markov chain.
\end{definition}

\noindent
In this section we will frequently use Airy functions. We recall basic
facts concerning these functions first.

\begin{remark}[Airy functions]\label{rem:airy}
  The Airy functions $Ai$ and $Bi$ are two linearly independent solutions of 
  the differential equations 
  \begin{align}
    \label{eq:4}
    u''(x) - x u(x)=0
  \end{align}
  with $Ai(x)\xrightarrow{x\to\infty} 0$ and
  $Bi(x)\xrightarrow{x\to\infty} \infty$. We only need properties of
  the Airy functions on the non-negative real line. For further
  properties and explicit definitions of the Airy functions in terms
  of integrals or Bessel functions we refer the reader to
  \cite{AbramowitzStegun:1972}.

  Denoting the gamma function by $\Gamma$ we have \citep[see][10.4.4
  and 10.4.5]{AbramowitzStegun:1972}
  \begin{align} \label{eq:54}
    Ai(0) & = \frac{Bi(0)}{\sqrt 3}
    = \frac{1}{3^{2/3}\Gamma(2/3)}
    \\ 
    \label{eq:55} 
    - Ai'(0) &  =\frac{Bi'(0)}{\sqrt 3}  = \frac{1}{3^{1/3}\Gamma(1/3)}.     
  \end{align}
  The Wronskian, 
  \begin{align}\label{eq:54a}
    Bi'(x)Ai(x)-Ai'(x)Bi(x)
  \end{align}
  which is the determinant of the fundamental matrix of the
  differential equation \eqref{eq:4}, does not depend on $x$ and is
  given by
  \begin{align} \label{eq:56} w:=Bi'(0)Ai(0)-Ai'(0)Bi(0) & =
    \frac2{\sqrt{3} \Gamma(\frac13)\Gamma(\frac23)} = \frac1\pi.
  \end{align}
  (The last equality follows from Euler's reflection formula
  $\Gamma(z)\Gamma(1-z)=\pi/\sin \pi z$.)

  The asymptotics of the Airy functions are well known, see e.g.\
  \citet[][p.~448--449]{AbramowitzStegun:1972} or
  \citet[][p.~161]{Janson:07}.  As $x \to \infty$ we have
  \begin{align} \label{eq:57} Ai(x) & \sim \frac{\pi^{-1/2}}2 x^{-1/4}
    e^{-2x^{3/2}/3}.
  \end{align}
  (Here, as usual, $a(x)\sim b(x)$ as $x \to \infty$ means that
  $a(x)/b(x)\to 1$ as $x \to \infty$.) The integral of $Ai$ is given
  by
  \begin{align} \label{eq:59} \int_0^\infty Ai(u)\,du = \frac13.
  \end{align}
  We note that 
  \begin{align} \label{eq:58} Gi(x) := Ai(x)\int_0^xBi(y)\,dy +
    Bi(x)\int_x^\infty Ai(y)\,dy.
  \end{align}
  is solution of the inhomogeneous equation
  \begin{align*}
    u''(x)-xu(x)=-\frac1\pi.
  \end{align*}

\end{remark}

\subsection{Increments for single jumps of the Brownian ratchet}
Between jumps of $\mathcal R$, the process $\mathcal X$ behaves like reflected
Brownian motion. A jump time of the ratchet can be seen as a killing time of
the reflected Brownian motion. Hence, we study a reflected Brownian motion,
killed at rate $\gamma$ times its value. Recall our assumption
$\gamma=\tfrac12$.  

\begin{definition}[Killed Brownian motion]
  \label{def:killedBrownian}
  Let $(B_t)_{t\ge 0}$ denote Brownian motion started in $x \ge 0$ and
  let
  \[\tau = \inf\Big\{t>0: \tfrac 12 \int_0^t\abs{B_s} \, ds 
  \geq \xi\Big\},\] where $\xi$ is an exponentially distributed,
  independent random variable with rate $1$. Then,
  $\widetilde{\mathcal B}:=(\widetilde B_t)_{t\geq 0}$ with
  $\widetilde B_t:=|B_{t}|$ for $0\leq t < \tau$ and $\widetilde
  B_t=\Delta$ for $t\geq \tau$ and some $\Delta\notin \mathbb R$ is a
  \emph{reflected Brownian motion, killed at rate $\tfrac 12 |B|$}. We 
  denote the probability measure of the Brownian motion, started in
  $x$, by $\Pr_x$.
\end{definition}

\noindent
To compute the speed of the ratchet we need the expectations of $\tau$
and $B_{\tau}$.

\begin{lemma}\label{lem:Btaux}
  Let $\widetilde{\mathcal B}$ be reflected Brownian motion, killed at
  rate $\tfrac 12 |B|$. Then,
  \begin{align}
    \label{eq:48}
    \E_x[\widetilde B_{\tau-}] & = x + \frac{2\pi Ai(x)}{3^{1/6}\Gamma(2/3)}  \\
    \intertext{and}
    \label{eq:49}
    \E_x[\tau] & = 2\pi (Gi(x)+3^{-1/2}Ai(x)).
  \end{align}
\end{lemma}

\begin{proof}
  First we compute the Green function of $\widetilde{\mathcal B}$
  following the scheme outlined on p.~18--19 in
  \cite{BorodinSalminen:2002}. Set $\sigma_y:= \inf\{t>0: \widetilde
  B_t = y\}$. The diffusion process $\widetilde{\mathcal B}$ is
  regular in the sense that for each $x,y \ge 0$
  \begin{align*}
    \Pr_x\big(\sigma_y < \infty\big) >0.
  \end{align*}
  Such a diffusion is called transient if for some (and then for all)
  $x,y, x\neq y$
  \begin{align*}
    \Pr_x\big(\sigma_y = \infty\big) > 0.
  \end{align*}
  As $\widetilde{\mathcal B}$ can be killed in each interval with
  positive probability it is transient. Let $p(.;.,.)$ be the
  transition density of $\widetilde{\mathcal B}$ with respect to the
  speed measure, which is given by $m(dx)=2\,dx$. In the transient
  case the Green function, defined by
  \begin{align*}
    G(x,y):=\int_0^\infty p(t;x,y)\,dt,
  \end{align*}
  is finite for each $x,y \ge 0$ and can be computed in terms of
  solutions of the differential equation \eqref{eq:4}
  with appropriate boundary conditions at $0$ and $\infty$.
  Two linearly independent solutions are given by the Airy functions
  $Ai$ and $Bi$, see Remark~\ref{rem:airy}. The solutions of
  \eqref{eq:4} that we need are obtained as follows. For $x\ge 0$ let
  \begin{align*}
    \phi(x)& =Ai(x) \\
    \psi(x)& =C_1 Bi(x) +C_2 Ai(x).
  \end{align*}
  The function $\phi$ is a decreasing solution of \eqref{eq:4} and
  satisfies 
  $\lim_{x\to\infty}\phi(x)=0$. The constants $C_1$ and $C_2$ have to
  be chosen such that the function $\psi$ is increasing and (this is
  the requirement for reflecting boundary at zero)
  \begin{align} \label{eq:26} 
    \psi'(0)=0 
  \end{align}
  It is easy to see, that in order to satisfy \eqref{eq:26} we must
  have $C_2= -C_1Bi'(0)/Ai'(0) = C_1\sqrt{3}$ (see \eqref{eq:55} for
  the last equality). We may take any positive number $C_1$ in the
  definition of $\psi$ because, as one can see in \eqref{eq:green},
  multiplication of $\phi$ and $\psi$ by some factors does not change
  the Green function. So we set $C_1=1$ and $C_2=\sqrt{3}$. The
  Wronskian (which is independent of $x$) is given as in \eqref{eq:56}
  by
  \begin{align}
    \label{eq:23}
      w =\psi'(0)\phi(0) - \psi(0)\phi'(0) = Bi'(0)Ai(0)-Ai'(0)Bi(0) = \frac1\pi.
  \end{align}
  Now the Green function is given by
  \begin{align}\label{eq:green}
    G(x,y) :=
    \begin{cases}
       w^{-1}\psi(x)\phi(y) & : \; 0 \le x\le y\\
      w^{-1}\psi(y)\phi(x) & : \; 0 \le y \le x.
    \end{cases}
  \end{align}
  The density of the killing measure with respect to the speed measure
  is $\tfrac 12 y$, hence it is $k(y)=y$ with respect to Lebesgue
  measure \citep[see][p.~17]{BorodinSalminen:2002}. Furthermore the
  killing position of the Brownian motion started in $x$ has density
  \begin{align*}
    G(x,y) k(y) 
  \end{align*}
  with respect to Lebesgue measure \citep[see][p.~14]{BorodinSalminen:2002}.  
  Now we can compute the expected killing position starting from $x$:
  \begin{align} \label{eq:1} \E_x[\widetilde B_{\tau-}]= \int_0^\infty
    y^2 G(x,y) \,dy = \frac{1}{w} \left(\phi(x) \int_0^x y^2\psi(y)\,dy
      + \psi(x) \int_x^\infty y^2\phi(y)\,dy\right).
  \end{align}
  We compute the integrals separately.  Using the fact that Airy
  functions are solutions of the differential equation \eqref{eq:4} we
  have
  \begin{align*}
    \int^x u^2 Ai(u)\, du & = \int^x u Ai''(u) \, du = xAi'(x) - Ai(x) + C \\
    \int^x u^2 Bi(u)\, du & = \int^x u Bi''(u) = xBi'(x) -Bi(x) + C.
  \end{align*}
  Thus, 
  \begin{align*}
    \phi(x) \int_0^x y^2\psi(y)\,dy & = Ai(x) \big(x
    Bi'(x)-Bi(x)+Bi(0) + \sqrt{3} (x Ai'(x) - Ai(x) + Ai(0)\big) \\
    \intertext{and} \psi(x) \int_x^\infty y^2\phi(y)\,dy & =
    \big(Bi(x)+\sqrt{3}Ai(x)\big)\big(-x Ai'(x)+Ai(x)\big).
  \end{align*}
  In the last equality we have used the asymptotic relation \eqref{eq:57}.
  Substituting in \eqref{eq:1} and simplifying we get 
  \begin{align*}
    \E_x[\widetilde B_{\tau}] = \frac{1}{w}\Big(x
    h(x)+Ai(x)(Bi(0)+\sqrt{3}Ai(0))\Big),
  \end{align*}
  where
  \begin{align*}
    h(x) = (Ai(x) Bi'(x) - Ai'(x) Bi(x)) =w.
  \end{align*}
  Using \eqref{eq:54} we arrive at
  \begin{align}
    \label{eq:12}
    \E_x[\widetilde B_{\tau}] & = x + \frac{2Ai(x)Bi(0)}{w} = x+ \frac{2\pi
      Ai(x)}{3^{1/6}\Gamma(2/3)}
  \end{align}
  which shows \eqref{eq:48}. For \eqref{eq:49}, the expected jump time
  of the reflection boundary if started from $x$ is given
  by 
  \begin{align} \label{eq:5} 
    \begin{split}
      \E_x\big[\tau\big] & = \int_0^\infty \int_0^\infty p(t;x,y)
      \,dt \, m(dy) \\ & = 2\int_0^\infty G(x,y)\,dy \\ & = \frac 2w
      \Big( \phi(x)\int_0^x \psi(y) dy + \psi(x) \int_x^\infty \phi(y)
      dy\Big) \\ & = \frac2{w}(Gi(x) + 3^{-1/2}Ai(x)) \\ & = 2\pi
      (Gi(x)+3^{1/2} Ai(x)),
    \end{split}
  \end{align} 
  where we have used \eqref{eq:59} for the next to last equality and
  the definition of $Gi$ in \eqref{eq:58}.
\end{proof}

In the next section we will need boundedness of second moments of
$\tau$, which is a consequence of the following lemma. Note that one
also could use Kac's moment formula \citep[see e.g.][]{0962.60067} to
compute the second moment of $\tau$ directly.

\begin{lemma}[Bound on tail distribution of $\tau$] There exist finite
  positive constants $c$ and $C$ such that for any $x,t \ge 0$
  \label{lem:taubound}
   \begin{align*}
     \Pr_x(\tau>t) \le Ce^{-c t}.
   \end{align*} 
\end{lemma}
\begin{proof}
  We use the notation of Definition~\ref{def:killedBrownian}.
  Clearly,
  \begin{align}
    \label{eq:46}
    \Pr_x(\tau > t) \le \Pr_0(\tau>t),
  \end{align}
  so it is enough to prove the assertion in the case $x=0$ and we omit
  the subscript $0$ in the proof.  Using the Brownian scaling property
  $(B_s)_{s\geq 0} \stackrel d = (t^{1/2}B_{s/t})_{s\geq 0}$ for any
  $t>0$ and a usual change of variables, we have
  \begin{align*}
    \int_0^{t} \abs{B_s}\, ds\stackrel{d}{=} t^{3/2} \int_0^1
    \abs{B_s}\, ds.
  \end{align*}
  We set $\mathcal{B}=\int_0^1\abs{B_s}\,ds$
  and use $f_{\mathcal{B}}$ to denote the density of that random
  variable. By the definition of $\tau$
  \begin{align}\label{eq:401}
    \Pr(\tau > t) & = \Pr\big(\tfrac 12 t^{3/2} \mathcal{B} <
    \xi\big)
    = \int_0^\infty f_{\mathcal{B}}(x)e^{-\tfrac 12 t^{3/2}x} \, dx =
    \E\big[e^{-\tfrac 12 t^{3/2} \mathcal B}\big],
  \end{align}
  i.e., the tail probability $\Pr(\tau>t)$ is given by the Laplace
  transform of $\mathcal{B}$ evaluated at $\tfrac 12 t^{3/2}$.  We
  consider the Laplace transform evaluated at $\lambda>0$ first and
  then set $\lambda = \tfrac 12 t^{3/2}$. A series representation of
  the Laplace transform of $\mathcal B$ in terms of Airy functions
  goes back to \citet{Kac:1946}. Instead of using the formula for the
  Laplace transform directly we use the asymptotic expansions of
  $f_{\mathcal{B}}$ to study the behaviour of the Laplace transform
  for large $t$.  We write
  \begin{align}
    \label{eq:41}
    \E[e^{-\lambda\mathcal{B}}]= \int_0^r e^{- \lambda
      x}f_{\mathcal{B}}(x)\,dx + \int_r^\infty
    e^{-\lambda x}f_{\mathcal{B}}(x)\,dx
  \end{align}
  and estimate both integrals on the right hand side for appropriately
  chosen $r$. The asymptotic expansions of  
  $f_{\mathcal{B}}$ are \citep[see][p.~108]{Janson:07}
  \begin{align}
    \label{eq:38}
    f_{\mathcal{B}}(x)  & = \frac{\sqrt{6}}{\sqrt{\pi}} e^{-3x^2/2}
    \big(1 + \mathcal{O}(x^{-2})\big)  \; 
    \text{ as } x \to \infty.   
    \\ \intertext{and} 
    \notag
    f_{\mathcal{B}}(x) & = A
    e^{-a/x^2}\big(x^{-2}+\mathcal{O}(1)\big) \; \text{ as } x \to
    0.
  \end{align}
  Here, $A$ is a positive constant, $a=2\abs{a_1'}^3/27 \approx
  0.0783$ and $a_1'$ is the largest real zero of the derivative of the
  Airy function $Ai$.
  
  Set $g(x):= 2 x^2\log x$ for $x>0$ and $g(0):=0$. Note that $g$ is
  continuous and negative on $(0,1)$. We choose $r>0$ such that $g(x)
  \ge -a/2$ for $x \in [0,r]$. Then for some $\widetilde{A} \ge A$ we
  have for all $x \in [0,r]$
  \begin{align*}
    f_{\mathcal{B}}(x) \le \widetilde{A} e^{-\frac{a}{x^2}}x^{-2} = 
    \widetilde{A} e^{-\frac{a+g(x)}{x^2}} \le  \widetilde{A}
    e^{-\frac{a}{2x^2}}. 
  \end{align*}
  It is easy to see that the function $x \mapsto \lambda x +
  \frac{a}{2x^2}$ attains its maximum in $x^*=(a/\lambda)^{1/3}$. For
  $\lambda \ge a/r^3$ we have $x^* \le r$ and therefore
  \begin{align}
    \label{eq:40}
    \begin{split}
      \int_0^re^{-\lambda x} f_{\mathcal{B}}(x)\, dx & \le
      \widetilde{A} \int_0^r
      e^{-\lambda x- \frac{a}{2x^2}} \, dx \\
      & \le \widetilde{A} r e^{-\lambda x^*- \frac{a}{2(x^*)^2}} =
      \widetilde{A} r e^{-3a^{1/3}\lambda^{2/3}/2}.
    \end{split}
  \end{align}
  Now we estimate the second integral in \eqref{eq:41}. From the
  asymptotic expansion \eqref{eq:38} and continuity of
  $f_{\mathcal{B}}$ it follows that there exists a finite positive
  constant $\widehat{A}$ such that
  \begin{align*}
    f_{\mathcal{B}}(x) \le \widehat{A}\frac{\sqrt{6}}{\sqrt\pi}
    e^{-3x^2/2}. 
  \end{align*}
  Completion of the square in the exponent followed by a change of
  variables yield 
  \begin{align*}
    \frac{\sqrt{6}}{\sqrt\pi} \int_r^\infty e^{-\lambda
      x-\frac{3x^2}2}\, dx & = \frac{\sqrt{6}}{\sqrt\pi} \int_r^\infty
    e^{-\frac{(3x +\lambda)^2}6+\frac{\lambda^2}6}\, dx \\ & =
    \frac{\sqrt{6}}{\sqrt\pi} e^{\frac{\lambda^2}{6}}\frac1{3}
    \int_{\frac{3r+\lambda}{\sqrt 3}}^\infty e^{-\frac{y^2}2}\,dy =
    e^{\frac{\lambda^2}6}\Big(1-\Phi\Big(\frac{3r+\lambda}{\sqrt 3}\Big)\Big) \\
    & \le e^{\frac{\lambda^2}6} \frac{\sqrt{3}}{3r+\lambda}
    e^{-\frac12(\frac{3r+\lambda}{\sqrt 3})^2} =
    \frac{{\bar{A}}}{3r+\lambda} e^{-\lambda r}
  \end{align*}
  for some constant $\bar{A}$. Here $\Phi$ denotes
  the distribution function of the standard Gaussian distribution. In
  the next to last step we used the well known inequality $1-\Phi(x)
  \le x^{-1}e^{-x^2/2}$. Combining the last two displays we obtain
  \begin{align}
    \label{eq:44}
    \int_r^\infty e^{-\lambda x}f_{\mathcal{B}}(x)\,dx \le
    \widetilde{A}\frac{\bar{A}}{3r+\lambda} e^{-\lambda r}.
  \end{align}
  Substitution of the estimates \eqref{eq:40} and \eqref{eq:44} in
  \eqref{eq:41} shows that there exists finite positive constants $c$
  and $C$ such that 
  \begin{align*}
    \E[e^{-\lambda \mathcal B}] \le C e^{- c (2\lambda)^{2/3}}, \quad
    t\ge 0.
  \end{align*}
  Now the assertion of the lemma follows from the last equation and
  \eqref{eq:401},
  \begin{align*}
    \Pr(\tau > t) = \E\big[ e^{-\tfrac 12 t^{3/2} \mathcal B}\big]
    \le C e^{-c t}.
  \end{align*}
\end{proof}

\subsection{Existence and uniqueness of the invariant distribution at
  jump times}
The Markov chain $(\mathcal Y, \mathcal W, \eta)$ describes the
Brownian ratchet at jump times. Next, we show existence (Proposition
\ref{P:exInv}) and uniqueness (Proposition \ref{P:coupl}) of the
invariant distribution for this Markov chain.

\begin{proposition}\label{P:exInv}
  Denote by $(P^n_x)_{n \in \N}$ the distribution of $(\mathcal Y, \mathcal W, 
  \eta)$ (introduced in Definition~\ref{def:mc}) induced by a Brownian
  ratchet starting with $(X_0,R_0)=(x,0)$. For each $x \ge 0$ there
  exists $i_1, i_2,\dots$ and a distribution $P^x$ that is an invariant 
  distribution for $(\mathcal Y, \mathcal W, \eta)$ and
   \begin{align} \label{eq:22} \frac1{i_n} \sum_{k=1}^{i_n} P^k_x 
    \xRightarrow{n\to\infty} P_x,
  \end{align}
  where `` $\Rightarrow$'' denotes weak convergence of probability
  measures.
\end{proposition}

\begin{proof}
  If the limit of a subsequence of $n^{-1} \sum_{k=0}^{n-1}P_x^n$
  exists then it must be an invariant distribution of
  $(Y_n,W_n,\eta_n)_{n \ge 1}$. A proof of this fact in the continuous
  time case, that can be easily adapted to the discrete time case, can
  be found in \citep[][p.~11]{Liggett:85}.  It remains to prove that
  $\Big(n^{-1} \sum_{k=0}^{n-1}P_x^n\Big)_{n=1,2,\dots}$ is a tight
  sequence. This follows immediately once we prove tightness of
  $(P_x^n)_{n=1,2,\dots}$. To this end it is enough to show that the
  first moments of $Y_n$, $W_n$ and $\eta_n$ are bounded uniformly in
  $n$.

  Boundedness of the first moments of $\eta_n$ follows from
  Lemma~\ref{lem:taubound}. For the boundedness of
  the first moments of $Y_1, Y_2,\dots$ and $W_1, W_2,\dots$, recall
  that $R_{\tau_n}$ is distributed uniformly on $[R_{\tau_n-};
  X_{\tau_n-}]$. Hence $Y_n$ and $W_n$ have the same distribution and
  it suffices to show boundedness of
  the first moment of $Y_1, Y_2,\dots$. For this, we recall from
  \eqref{eq:48} that
  \begin{align}\label{eq:27}
    \E_x\big[X_{\tau_1}\big] = \E_x[\widetilde B_{\tau-}] & = x +
    \frac{2\pi Ai(x)}{3^{1/6}\Gamma(2/3)} \leq x + c
  \end{align}
  for
  \[c=\frac{2\pi Ai(0)}{3^{1/6}\Gamma(2/3)},\] since $Ai$ is
  decreasing in $[0;\infty)$. Let $\FF_n=\sigma(X_{\tau_m},R_{\tau_m}:
  m \le n)$ and let $(U_n)$ be a sequence of iid random variables
  uniformly distributed on $(0,1)$ and independent of
  $(\mathcal{X},\mathcal{R})$. Using that sequence we have
  \begin{align*}
    X_{\tau_{n+1}}-R_{\tau_{n+1}} \stackrel{d}{=}  X_{\tau_{n+1}} - \big(
    R_{\tau_n} + (X_{\tau_{n+1}} - R_{\tau_n}) U_{n+1}\big) =
    (X_{\tau_{n+1}}-R_{\tau_n})(1-U_{n+1}). 
  \end{align*}
  and hence
  \begin{align*}
    \E_x\big[Y_{n+1}\big] & =
    \E_x\big[\E_x\big[(X_{\tau_{n+1}}-R_{\tau_{n}})(1-U_{n+1})|\FF_n\big]
    \big] \\ & = \frac12\E_x \big[\E_x
    \big[(X_{\tau_{n+1}}-R_{\tau_{n}})|\FF_n\big] \\ & = \frac 12
    \E_x\big[\E_{X_{\tau_n}}\big[X_{\tau_1} \big]-R_{\tau_n}\big]
    \\ & \leq \frac12 \E_x\big[Y_n+c\big] \\ & = \frac12
    \E_x\big[Y_n\big] + \frac12c
  \end{align*}
  where we have used \eqref{eq:27}. In other words, $\E_x[Y_n]\leq
  x+c$ for all $n=1,2,\dots$ which implies boundedness of the first
  moments of $Y_1, Y_2,\dots$.
\end{proof}

The following lemma shows uniqueness of the invariant distribution of
the Markov chain $(\mathcal Y, \mathcal W, \eta)$.

\begin{proposition} 
  \label{P:coupl}
  Let $(\mathcal Y, \mathcal W, \eta)$ be the Markov chain at jump
  times as introduced in Definition~\ref{def:mc}, based on a Brownian
  ratchet started in $x\geq 0$. If an invariant distribution of
  $(\mathcal Y, \mathcal W, \eta)$ exists, then it is unique.
\end{proposition}
\begin{proof}
  We use the graphical construction and a coupling argument for the
  proof. Let $N^\gamma$ be a Poisson process on $\mathbb R_+ \times
  \mathbb R$ with intensity measure $\gamma\cdot \lambda^2$. In
  addition, $\mathcal B^i=(B_t^i)_{t\geq 0}$, $i=1,2$ are two Brownian
  motions, started in $x_1\leq x_2$, such that $\mathcal B^1$ and
  $\mathcal B^2$ are independent before $T:= \inf\{s\geq 0: B_s^1 =
  B_s^2\}$ and $B_t^1=B_t^2$ for $t\geq T$. Moreover, $({\mathcal
    X}^{i}, {\mathcal R}^{i})$ is the Brownian ratchet read off from
  $(\mathcal B^i, N^\gamma)$, $i=1,2$.  Let $\tau_n^i$ denote the jump
  times of ${\mathcal R}^i$, $i=1,2$ and set
  \begin{align*}
    Y_{n}^i = X_{\tau_n^i} - R_{\tau_n^i}, \qquad W_n^i =
    R_{\tau_n^i} - R_{\tau_{n-1}^i} \qquad \eta_n^i= \tau_n^i-
    \tau_{n-1}^i
  \end{align*}
  for $i=1,2$. In order to show uniqueness of the invariant
  distribution of $(\mathcal Y, \mathcal W, \eta)$ it suffices to show
  that there is an almost surely finite stopping time $T'$ such that
  $(X_t^1-X_T^1, R_t^1-R_T^1)_{t\geq T'} = (X_t^2-X_T^2,
  R_t^2-R_T^2)_{t\geq T'}$, almost surely.

  We denote by $S_t^i$ the reflection point of $\mathcal B^i$ in the
  graphical construction, $i=1,2$, i.e., given $t\in[\tau_n^i;
  \tau_{n+1}^i), S_t^i := S_n^i$ where $S_1^i, S_2^i,\dots$,
  $\tau_1^i, \tau_2^i,\dots$ are as in \eqref{eq:tildetau} and
  \eqref{eq:Sn} for $i=1,2$. Since the hitting time of $\mathcal B^1$
  and $\mathcal B^2$ is almost surely finite, it is enough to
  construct the coupling using only one Brownian motion $B_t$ starting
  in $x=x_1=x_2$. So, assuming $B_0^1 = B_0^2$, we need to show that
  $T:=\inf\{t: R_t^1=R_t^2\}$ is almost surely finite. We set
  $s_i:=S_0^i$, $i=1,2$ and (without loss of generality) $s_1 \geq
  s_2$. In the case $s_1=s_2$ we have $T=0$ and we are
  done. Otherwise, there are three possibilities:
  \begin{align*}
    x \leq s_2< s_1, \quad s_2 < x < s_1,\quad s_2< s_1 \leq x,
  \end{align*}
  where the first and the last are symmetric. As the Brownian motion
  is recurrent and the reflection boundaries tend to jump towards the
  Brownian motion, it is clear that the time until the Brownian motion
  hits one of the moving boundaries is a.s.\ finite. Thus, it suffices
  to consider the cases
  \begin{align}\label{eq:twocases}
    x= s_2 \le s_1\quad \text{and} \quad s_2\le x = s_1.
  \end{align}
  Given $t\in[\tau^i_n;\tau^i_{n+1})$, recall that $S_t^i$ jumps at
  $\tau_{n+1}^i$ to the next point in $N^\gamma$ between $S_t^i$ and
  $B_t^i$, $i=1,2$. Hence, the processes couple if and only if both
  components of $(B_t,S_t^1)$ and $(B_t,S_t^2)$ use the same point of
  the Poisson process $N^\gamma$.
  This happens if the Poisson point, say $N(t,y)$, that is used in the
  construction satisfies
  \begin{align} \label{eq:6} B_t< N(t,y)< S_t^2\le S_t^1\quad
    \text{or} \quad S_t^2\le S_t^1< N(t,y) < B_t.
  \end{align}
  Consider the two cases of \eqref{eq:twocases}. In the first case we
  consider the (almost surely finite time of the) first jump of
  $S_t^2$ and in the second the (almost surely finite time of the)
  first jump of $S_t^1$. Due to the symmetry of the Brownian motion
  and the homogeneity of the Poisson process, this point is below $x$
  in the first case and above $x$ in the second case, with probability
  at least $\tfrac 12$. In other words, \eqref{eq:6} is satisfied at
  that jump time with probability at least $\tfrac 12$. If coupling
  did not occur, we wait an almost surely finite amount of time until
  either of \eqref{eq:twocases} is fulfilled again. Of course it is
  possible that one or both boundaries jump in-between. But at each
  jump time the boundaries get closer together. Then the processes try
  to couple again upon the next jump times of either $S_t^1$ or
  $S_t^2$ with success probability larger or equal $\tfrac 12$ and so
  on. Thus, the coupling time is bounded by a geometric number of
  finite times and hence is almost surely finite.
\end{proof}

\subsection{Increments under the invariant distribution}
\label{P:Btaupi}
Propositions~\ref{P:exInv} and \ref{P:coupl} guarantee existence and
uniqueness of the invariant distribution of the Brownian ratchet at
jump times. Next, we derive the increment of the Brownian ratchet and
waiting times between jumps in equilibrium.

\begin{proposition}
  \label{P:incrInv}
  Let $\pi$ be the invariant distribution for $Y_1, Y_2,\dots$. Start
  the Brownian ratchet in $\pi$ and denote the integral with respect
  to the resulting distribution by $\E_\pi$. Then,
  \begin{align}\label{eq:27a}
    \E_\pi[Y_1] =-3Ai'(0)
  \end{align}
  and
  \begin{align}\label{eq:27b}
    \E_\pi[\tau_1] =  6Ai(0).
  \end{align}
\end{proposition}

\begin{proof}
  Let $Y$ be distributed according to $\pi$.  Let a Brownian ratchet
  start in $X_0=Y, R_0=0$. Since $Y$ is in equilibrium, we find that
  the distribution of a Brownian ratchet at the first jump time is
  distributed like $Y$. To be more precise, let $\widetilde{\mathcal
    B}^Y = (\widetilde B_t^Y)_{t\geq 0}$ be a reflected Brownian
  motion, killed at rate $\tfrac 12 |B|$, started in $Y$ with
  increments independent of $Y$. Furthermore $\tau$ is the killing
  time of $\widetilde{\mathcal B}^Y$ as in
  Definition~\ref{def:killedBrownian}. Then, 
  \begin{align} \label{eq:3} Y \stackrel{d}{=}\widetilde B^Y_{\tau-}U,
  \end{align}
  where $U$ is independent and uniformly distributed on $(0,1)$, since
  the ratchet jumps to a uniformly distributed point in $[0;\widetilde
  B^Y_{\tau-}]$.

  We denote by $f_H$ the density of a random variable $H$. The
  requirement \eqref{eq:3} implies
  \begin{align*}
    f_Y(z) & =f_{B^Y_{\tau}U}(z) = \int_0^\infty f_Y(x)
    \int_0^\infty f_{B^x_{\tau}}(u)f_U\Big(\frac{z}u\Big)\frac1u\,du \,dx \\
    & = \int_0^\infty f_Y(x) \int_z^\infty
    f_{B^x_{\tau}}(u)\frac1u\,du\,dx
    \\ & =\int_0^\infty f_Y(x) \int_z^\infty G(x,u)\,du\,dx \\
    & = \frac{1}{w} \int_0^\infty f_Y(x)\Big[ \ind{z<x}\Big(
    \int_z^x\phi(x)\psi(u)\,du +
    \int_x^\infty\psi(x)\phi(u)\,du\Big) \\
    & \qquad \qquad \qquad \qquad \qquad \qquad \qquad \qquad \qquad +
    \ind{z\ge x} \int_z^\infty
    \psi(x)\phi(u)\,du\Big]\,dx\\
    & = \frac{1}{w} \Big[\int_z^\infty f_Y(x)\phi(x)
    \int_z^x\psi(u)\,du\,dx + \int_z^\infty f_Y(x)\psi(x)
    \int_x^\infty\phi(u)\,du \,dx  \\
    & \qquad \qquad + \int_0^z f_Y(x) \psi(x) \int_z^\infty
    \phi(u)\,du\,dx\Big] \\ &=: \frac1w[I_1(z) + I_2(z) + I_3(z)].  
  \end{align*}
  We differentiate with respect to $z$ and obtain
  \begin{align*} 
    I_1'(z) & = -\psi(z) \int_z^\infty f_Y(x)\phi(x) \, dx, \\
    I_2'(z) & = f_Y(z)\psi(z)\int_z^\infty \phi(u) \, du,  \\
    I_3'(z) & = -\phi(z)\int_0^zf_Y(x)\psi(x) \, dx -f_Y(z) \psi(z)\int_z^\infty
    \phi(u) \, du. 
  \end{align*}
  Thus,  
  \begin{align} \label{eq:10} f_Y'(z) = -\frac{1}{w} \Big[ \psi(z)
    \int_z^\infty f_Y(x)\phi(x)\, dx + \phi(z)
    \int_0^zf_Y(x)\psi(x)\,dx\Big].
  \end{align}
  At this point we see that $f_Y$ is strictly decreasing on $[0,\infty)$ since
  $\phi$ and $\psi$ are positive. We differentiate this equation two more
  times and obtain  
  \begin{align}
    \label{eq:11}
    \begin{split}
      f_Y''(z) & = -\frac{1}{w} \Big[\psi'(z) \int_z^\infty
      f_Y(x)\phi(x)\, dx - \psi(z)f_Y(z)\phi(z) \\ & \qquad \quad
      \qquad \quad \qquad \quad + \phi'(z) \int_0^z f_Y(x) \psi(x)\,dx
      +
      \psi(z)f_Y(z)\phi(z)\Big] \\
      & = -\frac{1}{w} \Big[\psi'(z) \int_z^\infty f_Y(x)\phi(x)\, dx
      + \phi'(z) \int_0^z f_Y(x) \psi(x)\,dx\Big]
    \end{split} \\ \intertext{and}
    \label{eq:13}
    \begin{split}
      f_Y'''(z) & = -\frac{1}{w} \Big[ \psi''(z) \int_z^\infty
      f_Y(x)\phi(x)\, dx -\psi'(z) f_Y(z)\phi(z) \\ & \qquad \quad
      \qquad \quad \qquad \quad +\phi''(z)
      \int_0^zf_Y(x)\psi(x)\,dx+\phi'(z)f_Y(z)\psi(z)\Big] \\ & =
      -\frac{1}{w} \Big[ z \psi(z) \int_z^\infty f_Y(x)\phi(x)\, dx +
      z\phi(z) \int_0^zf_Y(x)\psi(x)\,dx \\ & \qquad \quad \qquad
      \quad \qquad \quad - f_Y(z)(\psi'(z) \phi(z) - \phi'(z)\psi(z))
      \Big] \\ & = z f_Y'(z) + \frac{1}{w} f_Y(z)w = (z f_Y(z))'.
    \end{split}
  \end{align}
  Here we have used \eqref{eq:10} for the next to last
  equality. Integrating the last equation we get
  \begin{align}
    \label{eq:14}
    f_Y''(z)= z f_Y(z) + c
  \end{align}
  for some constant $c$. Using $\psi'(0) = 0$, which holds by the
  definition of $\psi$, and \eqref{eq:11} we obtain $c=f_Y''(0) =
  0$. From \eqref{eq:10} we know that $f_Y$ is decreasing. The unique
  (up to a constant factor) decreasing non-negative solution of
  \eqref{eq:14} with $c=0$ is given by $\phi=Ai$.  Thus, using
  \eqref{eq:59}, we have
  \begin{align}
    \label{eq:15}
    f_Y(z) = 3 \phi(z) = 3 Ai(z), z \ge 0.
  \end{align}
  The expectation of $Y$ is given by
  \begin{align}
    \label{eq:16}
    \E_{\pi}\big[Y_1\big] & = \int y \pi(dy) = \int y f_Y(y) dy = 3
    \int_0^\infty y Ai(y) \, dy = 3 \int_0^\infty Ai''(y)\, dy = - 3
    Ai'(0). 
  \end{align}
  Using \eqref{eq:49} we have
  \begin{align}\label{eq:17}
    \E_\pi[\tau_1] = \E[\E_Y[\tau]\big] = 2\pi \int_0^\infty f_Y(x)
    (Gi(x)+3^{-1/2}Ai(x))\,dx.
  \end{align}
  For evaluating the last integral we need
  \begin{align*}
    & \int_0^\infty Ai^2(x) \, dx = \Big[ x
    Ai^2(x)-(Ai'(x))^2\Big]_0^\infty = (Ai'(0))^2,  \\
    & \int_0^\infty Ai^2(x)\int_0^x Bi(u)\,du \,dx = \int_0^\infty
    Bi(u) \int_u^\infty Ai^2(x)\,dx du \\ & \qquad \qquad =
    \int_0^\infty Bi(u) \Big[ x Ai^2(x)-(Ai'(x))^2\Big]_u^\infty
    \,du \\ & \qquad
    \qquad = \int_0^\infty - u Ai^2(u) Bi(u) + (Ai'(u))^2 Bi(u) \, du,\\
    & \int_0^\infty Ai(x)Bi(x)\int_x^\infty Ai(u)\,du \,dx =
    \int_0^\infty Ai(u) \int_0^u Ai(x)Bi(x)\,dx du \\ & \qquad \qquad
    = \int_0^\infty Ai(u) \Big[ x Ai(x)Bi(x)-Ai'(x)Bi'(x)\Big]_0^u
    \,dx \\ & \qquad \qquad = \int_0^\infty u Ai^2(u) Bi(u) - Ai(u)
    Ai'(u) Bi'(u) + Ai(u) Ai'(0) Bi'(0) \, du.
  \end{align*}
  Plugging \eqref{eq:15} into \eqref{eq:17} and combining the last
  equations, we obtain, using \eqref{eq:55}, \eqref{eq:59} and the
  fact that the Wronskian \eqref{eq:54a} does not depend on $x$,
  \begin{align*}
    \E_\pi[\tau_1] & = 6\pi \Big( \int_0^\infty Ai(x) Gi(x)\, dx +
    3^{-1/2} \int_0^\infty (Ai(x))^2\, dx\Big) \\ & = 6\pi\Big(
    \int_0^\infty Ai'(u) (Ai'(u) Bi(u)-Bi'(u) Ai(u) ) \, du \\ & 
    \qquad \qquad \qquad \qquad 
    \qquad + Ai'(0)
    Bi'(0) \int_0^\infty Ai(u)\, du  + 3^{-1/2} (Ai'(0))^2\Big) \\ & = 6\pi\Big( - w
    \int_0^\infty Ai'(u)\, du + \frac13Ai'(0) Bi'(0) - \frac 13
    Ai'(0)Bi'(0)\Big) \\ & = 6 Ai(0).
  \end{align*}
\end{proof}

\section{Regeneration structure}
\label{sec:regen-struct}
In this section we show that the Brownian ratchet has a renewal
structure and can be seen as a cumulative process with a remainder
term (see Definition~\ref{def:cumu1}). In order to use this structure
we will have to bound moments of several quantities, i.e.\ of the time
between renewal points, (Proposition~\ref{P:mombounds1}), and the
increment of the Brownian ratchet at renewal points
(Proposition~\ref{P:mombounds2}).  As in the last section, we fix
$\gamma=\tfrac 12$.

\subsection{The Brownian ratchet as a cumulative process}

\begin{remark}[Cumulative processes]
  \label{rem:cumu1}
  We recall a definition of cumulative processes from
  \citet{Roginsky:1994}. Let $(T_n,V_n)_{n \ge 1}$ be a sequence of
  bivariate iid random variables with $T_1>0$ almost surely. The times
  $T_1, T_2,\dots$ are called regeneration times. Define a renewal
  process $(S_{M_t})_{t \ge 0}$ by
  \begin{align*}
    M_t = \min\Big\{n:\sum_{i=1}^n T_i > t\Big\} \quad \text{ and
    }\quad S_n=\sum_{i=1}^{n} V_i.
  \end{align*}
  The process $(S_{M_t})_{t\geq 0}$ is called a type A cumulative
  process. Cumulative processes are well studied \citep[see e.g.][and
  references therein]{Smith:1955,Roginsky:1994}. Most importantly, it
  is known that finite second moments of $T_1$ and $V_1$ are
  sufficient for the strong Law of Large Numbers and the Central Limit
  Theorem for $(S_{M_t})_{t \ge 0}$.
\end{remark}

\begin{definition}[The Brownian ratchet as a cumulative process]
  \label{def:cumu1}
  Given a Brownian ratchet $(\mathcal X, \mathcal R) = (X_t,
  R_t)_{t\geq 0}$ with $(X_0,R_0)=(x,0)$, $x \ge 0$, we define a
  sequence of regeneration times using times when $X_t=R_t$ as well as
  jump times of $\mathcal R$. We set
  \begin{align}\label{eq:501}
    \rho_0 & :=\inf\{t \ge 0: X_t = R_t\}, \\
    \tilde\rho_0 & := \inf\{t > \rho_0: R_t \ne R_{t-}\}, \notag\\
    \intertext{and for $n =1,2,\dots$, using $R_{t-}= \lim_{s \to t,
        s<t} R_s$}\label{eq:502}
    \rho_{n} & := \inf\{ t > \tilde\rho_{n-1}:X_t=R_t\}, \\
    \tilde\rho_{n} & := \inf\{t > \rho_{n}: R_t \ne R_{t-}\}\notag
  \end{align}
  such that $\rho_0\leq \tilde\rho_0\leq\rho_1\leq\tilde\rho_1\leq
  \dots$. Note that if $x=0$, then $\rho_0=0$ and $X_{\rho_0}=0$.  It is 
  clear that $(\rho_n-\rho_{n-1}, X_{\rho_n}-X_{\rho_{n-1}})_{n\ge 1}$
  is a sequence of bivariate iid random variables. We define
  \begin{align*} 
    M_t & :=\min\{n:\rho_n > t\}, \\
    S_n & :=\sum_{i=1}^{n}(X_{\rho_i} - X_{\rho_{i-1}}),\\
    A_t & := X_{\rho_0} + X_t - X_{\rho_{M_t}}.
  \end{align*}
  Then we have
  \begin{align} \label{eq:51} X_t = S_{M_t} + A_t,
  \end{align}
  that is, $\mathcal X$ is a type A cumulative process with remainder
  $A_t$.
\end{definition}

\begin{remark}[The regenerative structure in the proof of
  Theorem~\ref{T:LLN} and \ref{T:CLT}] As mentioned in
  Remark~\ref{rem:cumu1}, finite second moments of $\rho_1-\rho_0$ and
  $X_{\rho_1}-X_{\rho_0}$ are sufficient for the strong Law of Large
  Numbers and the Central Limit Theorem for $S_{M_t}$. To see that the
  same holds for $\mathcal X$, showing Theorems~\ref{T:LLN} and
  \ref{T:CLT}, we shall show in Proposition~\ref{prop:AsympA} that the
  remainder divided by $\sqrt{t}$ converges to $0$ almost surely as
  $t\to\infty$.
\end{remark}

\subsection{Tail estimates and moment bounds}
In the rest of the section, we denote the law of a Brownian ratchet
$(\mathcal X, \mathcal R)$, started in $(x,0)$ by $\Pr_x$ (with
expectation operator $\E_x$). We use the notation of
Definition~\ref{def:cumu1}. 
We can use the Law of Large Numbers and the Central Limit Theorem for
cumulative processes given $\rho_1-\rho_0$ and $X_{\rho_1}-X_{\rho_0}$
have a finite second moment. These are the main results of this section
(Proposition~\ref{P:mombounds1} and \ref{P:mombounds2}).

\begin{proposition}[Moment bounds for $\rho_1-\rho_0$]
  \label{P:mombounds1}
  Let $\rho_0, \rho_1$ be as in \eqref{eq:501} and
  \eqref{eq:502}. Then,
  \begin{align}
    \E_x[(\rho_1-\rho_0)^2] < \infty.
  \end{align}
\end{proposition}

\noindent The proof is based on Lemma~\ref{lem:taubound} and the following
result. 

\begin{lemma} \label{lem:rx} Let $\sigma$ be the stopping time when
  the Brownian ratchet hits the moving reflection boundary for the
  first time, i.e.
  \begin{align} \label{eq:511} \sigma := \inf\{t >0: X_t=R_t\}.
  \end{align}
  Then, there exists a finite positive constant $C$ such that
  \begin{align*}
    \sup_x \E_x[ \sigma^2] < C.
  \end{align*}
\end{lemma}

\begin{proof}[Proof of Proposition~\ref{P:mombounds1}]
  Clearly, $\rho_1-\rho_0 = (\tilde\rho_0 - \rho_0) + (\rho_1 -
  \tilde\rho_0)$. Since $\tilde\rho_0-\rho_0 \stackrel d = \tau$ with
  $\tau$ from Definition~\ref{def:killedBrownian}, finiteness of the
  second moment of the first summand follows immediately from
  Lemma~\ref{lem:taubound}. Moreover, given $X_{\tilde\rho_0} -
  R_{\tilde\rho_0} = z$, we have that $\rho_1 - \tilde\rho_0$ is
  distributed as $\sigma$ for a Brownian ratchet started in $z$ with
  $\sigma$ from \eqref{eq:511}. Hence, finiteness of the second moment
  of the second summand follows from Lemma~\ref{lem:rx}. Thus, the
  proposition is proved.
\end{proof}

\begin{proof}[Proof of Lemma~\ref{lem:rx}]
  Recall the jump times $\tau_1, \tau_2,\dots$ of $\mathcal R$. We
  consider the process $\mathcal Z := (Z_{t\wedge \sigma})_{t\geq 0}$
  for $Z_t:=X_t-R_t$ which is started in $x$, and has discontinuities
  at $\tau_1, \tau_2,\dots$.  Note that $\sigma$ is the time at which
  $Z_t$ hits $0$ for the first time.

  Note that $\mathcal Z$ is an autonomous Markov process: locally it
  behaves like Brownian motion, reflected at 0 and each time there is
  a point in a Poisson point process on $[0;\infty)\times\mathbb R$ in
  $[0,Z_t]$ the process $Z_t$ restarts from this Poisson point,
  i.e.\ at discontinuity points the process $Z_t$ jumps down.

  Our goal is to bound $\sigma$ from above by a simpler random
  variable $\hat\sigma$. To this end we couple $\mathcal Z$ to a
  process $\widehat{\mathcal Z} = (\widehat Z_t)_{t\geq 0}$ which
  starts in $x$ and has the following dynamics: the local increments
  of $\widehat{\mathcal Z}$ follow exactly those of $\mathcal Z$,
  i.e.\, $\widehat{\mathcal Z}$ and $\mathcal Z$ use the same
  underlying Brownian motion. At jump times $\tau_n$ of $\mathcal Z$
  with $Z_{\tau_n}\leq 1$, we set $\widehat Z_{\tau_n}=1$. In
  addition, $\widehat Z$ jumps at rate $(1-Z_t)^+$ to 1. At jump times
  $\tau_n$ with $Z_{\tau_n}>1$, the process $\widehat{\mathcal Z}$
  does not jump. We denote the first hitting time of
  $\widehat{\mathcal Z}$ of 0 by $\widehat \sigma$. By this coupling,
  we have achieved the following:
  \begin{enumerate}
  \item At joint jump times $\tau$ of $\widehat{\mathcal Z}$ and
    $\mathcal Z$ we have $Z_\tau \leq \widehat Z_\tau=1$. Jump times
    $\tau$ exclusive to $\widehat{\mathcal Z}$ must satisfy
    $Z_\tau\leq 1=\widehat Z_\tau$. For a jump time $\tau$ exclusive
    to $\mathcal Z$, and if $Z_{\tau-}\leq \widehat Z_{\tau-}$, we
    have $Z_\tau \leq Z_{\tau-} \leq \widehat Z_{\tau-} = \widehat
    Z_{\tau }$. In particular, for all $t\geq 0$ we have $ Z_t \leq
    \widehat Z_t$.
  \item The process $\widehat{\mathcal Z}$ jumps at rate 1. If $\tau$
    is a jump time of $\widehat Z$, then $\widehat Z_t=1$.
  \end{enumerate}
  In particular, $\sigma\leq \widehat \sigma$, almost surely, by 1.,
  and it suffices to show that $\widehat \sigma$ has a finite second
  moment. Set $\hat\tau_0=0$ and denote the jump times of
  $\widehat{\mathcal Z}$ by $\hat \tau_1, \hat \tau_2,\dots$. In
  addition, let $(B_t)_{t\geq 0}$ be Brownian motion starting in $1$,
  set
  \begin{align*}
    \tau= \inf\{t > 0: B_t=0\}
  \end{align*}
  and
  \begin{align*}
    T_i = \hat\tau_{i+1}-\hat\tau_{i}, \; i \ge 0.
  \end{align*}
  It is clear that $T_0,T_1,\dots$ are iid random variables with rate
  one independent of $\widehat Z_t$. Independent of initial position
  $\widehat Z_0$ the process $\widehat{\mathcal Z}$ restarts from $1$
  at time $T_1$. Define $p:= \Pr_1(\tau < T_1) $ and let $M$ be
  geometrically distributed with parameter $p$ independent of $T_1,
  T_2,\dots$, i.e.\ $\Pr(M = n)= (1-p)p^n$, $n \ge 0$. Note that $M$ is
  the number of fail trials for the process $\widehat{\mathcal Z}$ to
  reach $0$, started in $1$, before the next jump. It follows that
  \begin{align*}
    \widehat \sigma & \le \sum_{i=0}^{M} T_i + \widetilde \tau.
  \end{align*}
  where $\widetilde \tau$ is distributed as $\tau$, conditioned to
  be smaller than $T_1$ and hence
  \begin{align*}
    \E_1[\widetilde \tau^2] = \frac{\E_1[\tau^2, \tau\leq 
      T_1]}{\Pr_1(\tau<T_1)} \leq \frac{\E_1[T_1^2]}{p} = \frac{2}{p}.
  \end{align*}
  Combining these results,
  \begin{align*}
    \E_x[\sigma^2] & \leq \E_x[\widehat \sigma^2] \leq \E\Big[
    \Big(\E\Big[\sum_{i=0}^M T_i + \widetilde\tau\Big|M\Big]\Big)^2\Big] \\
    & \leq \E\Big[ \Big(2(M+1) + (M+1)M + \frac 1 p(3+M)\Big)^2\Big]
    \\ & =: C < \infty
   \end{align*}
   The right hand side is independent of $x$, and we are done.
\end{proof}

\begin{proposition}[Moment bounds for $X_{\rho_1}-X_{\rho_0}$]
  \label{P:mombounds2}
  We have
  \begin{align} \label{eq:50} \E_x[X_{\rho_0}^2] & < \infty,
    \\\label{eq:50b} \E_x[(X_{\rho_1}-X_{\rho_0})^2] & < \infty\\
    \intertext{and for each $t \ge 0$}
    \label{eq:30}
    \E_x[(X_t-X_{\rho_{M_t}})^2 ] & \le 2
    \big(\E_x[(X_{\rho_1}-X_{\rho_0})^2] + \E_x[X_{\rho_0}^2]\big).  
  \end{align}
\end{proposition}
\begin{proof}
  Note that $|X_{\rho_0}| \sim |B_{\rho_0}|$, which already shows that
  \[ \E_x[X_{\rho_0}^2] = \E_x[B_{\rho_0}^2] = \E_x[\rho_0] < \infty\]
  by the second Wald identity where we have used that $\E_x[\rho_0]< \infty$ by
  Lemma~\ref{lem:rx}.  
  
  Now, using the strong Markov property at time $\rho_0$ we have
  \begin{align*}
    \E_x[(X_{\rho_1}-X_{\rho_0})^2] = \E_0[X_{\rho_1}^2] = \E_0
    [B_{\rho_1}^2] =  \E_0[\rho_1] = \E_x[\rho_1-\rho_0] < \infty,
  \end{align*}
  where we have again used Wald's second identity that is applicable by 
  Proposition~\ref{P:mombounds1}. This proves \eqref{eq:50b}. 
  
   It remains to show \eqref{eq:30}. We have 
   \begin{align} \label{eq:19} 
     \E_x[(X_t-X_{\rho_{M_t}})^2 ] = \E_x[(X_t-X_{\rho_{M_t}})^2 \ind{t
       <\rho_0} ] + \E_x[(X_t-X_{\rho_{M_t}})^2 \ind{t \ge \rho_0} ].  
   \end{align}
   The first summand on the right hand side can be estimated as follows 
   \begin{align} \label{eq:20} 
     \begin{split}
       \E_x[(X_t-X_{\rho_0})^2 \ind{t <\rho_0} ] & \le \E_x[X_t^2 \ind{t <
         \rho_0}] + \E_x[X_{\rho_0}^2] \\ & \le \E_x[X_{t\wedge \rho_0}^2] + 
       \E_x[X_{\rho_0}^2] \le 2 \E_x[X_{\rho_0}^2], 
     \end{split}
   \end{align}
   where for the last inequality we used the fact that the stopped
   process $\big(X_{t \wedge \rho_0}^2\big)$ is a sub-martingale (as
   usual $a \wedge b = \min\{a,b\}$). Using a similar argument we have
   for the second summand on the right hand side of \eqref{eq:19}
   \begin{align}
     \label{eq:21} 
     \begin{split}
       \E_x[(X_t-X_{\rho_{M_t}})^2 \ind{t \ge \rho_0} ] &  =
       \E_0[(X_t-X_{\rho_1})^2 \ind{t < \rho_1} ] \le \E_0[X_t^2 \ind{t <
         \rho_1}] + \E_0[X_{\rho_1}^2]  \\ & \le 
       \E_0[X_{t \wedge \rho_1}^2] +  \E_0[X_{\rho_1}^2] \le 2
       \E_0[X_{\rho_1}^2] = 2\E_x[(X_{\rho_1}-X_{\rho_0})^2].  
     \end{split}
   \end{align} 
   Combining \eqref{eq:19}, \eqref{eq:20} and \eqref{eq:21} we obtain
   \eqref{eq:30}. 

 \end{proof}
 \begin{proposition}[Asymptotics of $A_t$ and $X_t-R_t$]  \label{prop:AsympA} 
   We have \[\frac{A_t}{\sqrt t} \to 0 \quad\text{ and } \quad
   \frac{X_t-R_t}{\sqrt{t}}\to 0 \quad \text{a.s.\ as $t \to \infty$}.\]  
 \end{proposition}
 \begin{proof}
   The proof of the first assertion is an adaptation of the proof of
   Lemma~8 in \citep[][p.~26]{Smith:1955}. In view of \eqref{eq:50} it
   is enough to prove the first assertion for $X_0=0$. Then we need to
   show that
   \begin{align}
     \label{eq:7}
     \frac{X_t-X_{\rho_{M_t}}}{\sqrt t} \to 0 \; \text{ a.s.\ as $t \to 
       \infty$}.  
   \end{align}
   We set
   \[Y_{n}=\sup_{t \in [\rho_{n-1},\rho_{n}]}\abs{X_t-X_{\rho_{n}}}.\]
   Then $Y_1,Y_2, \dots$ is a sequence of independent identically
   distributed random variables with finite second moments. To prove
   this we note that
   \begin{align*}
     Y_1 \le X_{\rho_1} +\sup_{t \ge 0} \abs{X_{t \wedge \rho_1}}.    
   \end{align*}
   By Doob's maximal inequality 
   \begin{align*}
     \E_0\Big[ \sup_{t \ge 0}\abs{X_{t \wedge \rho_1}}^2\Big] \le 4 \E_0[X_{\rho_1}^2].    
   \end{align*}  
   Thus, finiteness of second moments of $Y_1$ follows from \eqref{eq:50b}.  
   Now  
   \begin{align}
     \label{eq:8}
     \frac{\abs{X_t-X_{\rho_{M_t}} }}{\sqrt t} \le \frac{Y_{M_t}}{\sqrt t} =
     \frac{Y_{M_t}}{\sqrt {M_t}} \frac{\sqrt{M_t}}{\sqrt t}.  
   \end{align}
   Note that by Theorem~1 in \cite{0041.45405} we have
   \begin{align} \label{eq:9} \lim_{t\to\infty}\frac{M_t}{t} =
     \frac{1}{\E[\rho_1-\rho_0]}>0, \; \text{a.s.}
   \end{align} Using Lemma~7 in \citep[][p.~26]{Smith:1955} we obtain
   $Y_{M_t}/\sqrt {M_t} \to 0$ almost surely and therefore
   \eqref{eq:7} holds.

   For the second assertion we have
   \begin{align*}
     \frac{(X_t-R_t)\ind{t < \rho_0}}t \le \frac{X_{t \wedge \rho_0}}{t} \to 0
     \; \text{ a.s. as $t\to \infty$}.  
   \end{align*}
    For $t \in [\rho_{n-1},\rho_n]$, $n\ge 1$ we have \[0 \le X_t -R_t \le 
   Y_n.\] 
   Thus,
   \begin{align*}
     \frac{(X_t-R_t)\ind{t \ge \rho_0}}t \to 0
     \; \text{ a.s. as $t\to \infty$}    
   \end{align*}
   follows by the same argument as before.  
 \end{proof}

\section{Proofs of Theorem~\ref{T:LLN} and \ref{T:CLT}}
\label{sec:proof}
In the case $\gamma=0$, Theorem 1 is true since $X_t \stackrel d =
|B_t|$ for a Brownian motion, started in $x$. Hence, we can assume
$\gamma>0$ in the rest of the proof.

We make use of the regeneration structure, set out in
Definition~\ref{def:cumu1}. We set 
\begin{align} \label{eq:18} 
  r=\E_x[\rho_1-\rho_0], && \mu=\E_x[X_{\rho_1}-X_{\rho_0}] &&
  \text{and} &&
  \beta^2=\Var_x\Big[X_{\rho_1}-X_{\rho_0}-\frac{(\rho_1-\rho_0)\mu}r\Big].
\end{align}
Here, $r, \mu$ and $\beta^2$ are independent of $x$ due to the
regeneration structure. By Propositions~\ref{P:mombounds1} and
\ref{P:mombounds2} the time and space increments $\rho_1-\rho_0$ and
$X_{\rho_1}-X_{\rho_0}$ have finite second moments.  As in
\eqref{eq:9} we have $ \lim_{t\to\infty}\frac{M_t}{t} = \frac1r$
almost surely.

From Proposition~\ref{prop:AsympA}, \eqref{eq:9} and Law of Large
Numbers applied to $S_n$ it follows that
\begin{align*}
  \frac{X_t}t= \frac{A_t}t + \frac{M_t}{t} \frac{S_{M_t}}{M_t} \to
  \frac{\mu}r\; \text{a.s.} \quad \text{as $t \to \infty$.}
\end{align*}
We will compute $\mu/r$ using the ratio limit theorem for Harris
recurrent Markov chains \citep[see e.g.][]{Revuz:1984}. Let $\pi$
denote the invariant distribution for $(\mathcal Y, \mathcal W, \eta)$
determined by Propositions \ref{P:exInv} and \ref{P:coupl}. To show
that this Markov chain is Harris recurrent we have to show that each
Borel subset of $\R^3_{\ge 0}$, say $B$, having positive Lebesgue
measure is visited infinitely often with probability $1$. We cannot
apply the Borel-Cantelli lemma directly, because the events
$\{(Y_n,W_n,\eta_n) \in B\}_{n \ge 1}$ are not independent. But events
lying between different regeneration points are independent.  For $n
\ge 0$ we set \[B_n=\{\exists m \ge 1 \text{ s.th.\ } \rho_{n-1}
<\tau_{m-1}, \tau_m < \rho_{n}) \text{ and } (Y_m,W_m,\eta_m) \in B
\}.  \] The events $B_n$ are independent and have the same positive
probability. Thus, applying the Borel-Cantelli lemma to
the sequence $(B_n)_{n \ge 0} $ we obtain
\begin{align*}
  \Pr_x\big( (Y_n,W_n,\eta_n) \in B \; \text{i.o.}\big) \ge
  \Pr_x\big( B_n \; \text{i.o.}\big) = 1.
\end{align*}
Now using the ratio limit theorem we obtain that 
\begin{align*}
  \frac{\mu}{r} = \lim_{t\to\infty}\frac{X_t}t= \lim_{t\to\infty}
  \frac{R_t}t=\lim_{n\to\infty}\frac{R_{\tau_n}}{\tau_n} = \lim_{n \to
    \infty} \frac{\sum_{k=1}^n W_k}{\sum_{k=1}^n \eta_k} =
  \frac{\E_\pi[W_1]}{\E_\pi[\eta_1]}\; \text{ a.s. $n\to\infty$}, 
\end{align*} 
where for the second equality we have used Proposition~\ref{prop:AsympA}. 
We recall that $\E_\pi[W_1]=\E_\pi[Y_1]$.  Furthermore, recalling that
we have used $\gamma=\tfrac 12$ in Section \ref{sec:invar-distr}, and
using \eqref{eq:54}, \eqref{eq:55} as well as Proposition
\ref{P:incrInv}, we obtain
\begin{align*}
  \frac\mu{r} = \frac{\E_\pi[Y_1]}{\E_\pi[\tau_1]} = -(2\gamma)^{1/3}
  \frac{3 Ai'(0)}{6 Ai(0)} = \frac{\gamma^{1/3}}{2^{2/3}}
  \frac{3^{1/3}\Gamma(2/3)}{\Gamma(1/3)}.
\end{align*}
by the scaling property, Proposition~\ref{P:scale}. This proves Theorem
\ref{T:LLN}.  

Using the CLT for cumulative processes \citep[see
e.g.][]{Smith:1955,Roginsky:1994} and
Proposition~\ref{prop:AsympA} we obtain that for all $x \in
\R$ 
\begin{align*}
  \lim_{t \to \infty} \Pr\biggl(\frac{X_t - t\mu/r}{\beta
    (t/r)^{1/2}} \le x\biggr) & = \lim_{t \to \infty}
  \Pr\biggl(\frac{A_t}{\beta (t/r)^{1/2}} + \frac{S_{M_t} -
    t\mu/r}{\beta (t/r)^{1/2}} \le x\biggr) \\ & = \lim_{t \to
    \infty} \Pr\biggl(\frac{S_{M_t} - t\mu/r}{\beta (t/r)^{1/2}} \le
  x\biggr) = \Phi(x),
\end{align*}
where $\Phi$ denotes the distribution function of the standard normal
distribution, and $\beta^2$ and $r$ are as defined in \eqref{eq:18}.
Hence, Theorem 2 holds for $\sigma = \beta/\sqrt{r}$.

\subsection*{Acknowledgement} We thank Jens Timmer for bringing our
attention to models of Brownian ratchet in cellular biology and Heinz
Weisshaupt for fruitful discussion. We are grateful to Gerhard Keller
who pointed us to an incorrect argument in an earlier version and two
referees for helping us to improve our manuscript.


\end{document}